\documentclass{amsart}    
\usepackage{amsmath, amsthm, amssymb}
\usepackage[english]{babel}
\usepackage[T1]{fontenc}
\usepackage[utf8]{inputenc}
\usepackage{mathrsfs}
\usepackage{pgf}
\usepackage{pinlabel}

\title[Ray graph and quasimorphisms on a big mapping class group]{Gromov-hyperbolicity of the ray graph and quasimorphisms on a big mapping class group}

\author{Juliette Bavard}
\address{Univ Rennes, CNRS, IRMAR - UMR 6625, F-35000 Rennes, France}
\email{juliette.bavard@univ-rennes1.fr}

\newtheorem{theo}{Theorem}[section]
\newtheorem*{theo*}{Theorem}
\newtheorem*{prop*}{Proposition}
\newtheorem{prop}[theo]{Proposition}
\newtheorem{lemma}[theo]{Lemma}
\newtheorem{corollary}[theo]{Corollary}

\newtheorem{claim}[theo]{Claim}

\theoremstyle{definition}
\newtheorem*{definition}{Definition}
\newtheorem*{remark}{Remark}

\newcommand{\Sph}{\mathbb{S}}

\newcommand{\E}{\mathcal{E}}
\newcommand{\Cc}{\mathcal{C}}

\newcommand{\N}{\mathbb{N}}
\newcommand{\R}{\mathbb{R}}
\newcommand{\Z}{\mathbb{Z}}

\DeclareMathOperator{\Homeo}{Homeo}

\DeclareMathOperator{\MCG}{MCG}
\DeclareMathOperator{\Cantor}{Cantor}
\DeclareMathOperator{\scl}{scl}

\begin{document}

\maketitle

The goal of these notes is to translate\footnote{Any improvement/comment on this translation is welcome at juliette.bavard@univ-rennes1.fr.} \cite{Juliette}, which was written in French in September $2014$. 

\begin{abstract}   
The mapping class group $\Gamma$ of the complement of a Cantor set in the plane arises naturally in dynamics. We show that the ray graph, which is the analog of the complex of curves for this surface of infinite type, has infinite diameter and is hyperbolic. We use the action of $\Gamma$ on this graph to find an explicit non trivial quasimorphism on $\Gamma$ and to show that this group has infinite dimensional second bounded cohomology. Finally we give an example of a hyperbolic element of $\Gamma$ with vanishing stable commutator length. This carries out a program proposed by Danny Calegari.
 \end{abstract}

\section{Introduction}
\subsection{Big mapping class group and dynamics}

Let $S$ be a connected, orientable surface, which is not necessarily assumed to have finite topological type (i.e. $S$ can be a compact surface with infinitely many punctures, or a surface with infinite genus, etc). The \emph{mapping class group of $S$}, that we will denote by $\MCG(S)$, is the group of preserving orientation homeomorphisms of $S$ up to isotopy. If we know many characteristics of mapping class groups of finite type surfaces, those of infinite type surfaces have been less studied. However, as Danny Calegari explained in his blog "Big mapping class groups and dynamics"~\cite{blog-Calegari}, those "big" mapping class groups appear naturally in dynamical problems, in particular through the following construction (see \cite{blog-Calegari}).

Let us denote by $\Homeo^+(\R^2)$ the group of homeomorphisms of the plane which preserve the orientation. Let $G$ be a subgroup of $\Homeo^+(\R^2)$. If the orbit $G\cdot p$ of a point $p \in \R^2$ is bounded, then there exists a morphism from $G$ to $\MCG(\R^2- K)$, where $K$ is either a finite set, or a Cantor set.
Indeed, the union $\tilde K$ of the closure of the orbit $G \cdot p$ with the bounded connected components of its complement is a compact set, invariant by $G$, and whose complement is connected. The group $G$ acts on the quotient of the plane that we get by collapsing each of the connected components of $\tilde K$. This quotient is homeomorphic to the plane (by a theorem of Moore). The image of $\tilde K$ in the quotient is a subset $K$ of the plane, which is totally disconnected. Up to replace $K$ by one of its subset, we can assume that $K$ is minimal, i.e. every orbit $G \cdot q$ with $q\in K$ is dense in $K$. Because $K$ is compact, it is either a finite set, or a Cantor set. This construction gives us a morphism from $G$ to $\MCG(\R^2-K)$.

The mapping class group of $\R^2 $ minus finitely many points has a finite index subgroup isomorphic to a braid group quotiented by its center, and thus has been well studied. In this paper, we will focus on the other case, where $K$ is a Cantor set. We will denote: 
$$\Gamma:=\MCG(\R^2-\Cantor).$$ 

In \cite{Calegari-Circular}, Calegari proves that there exists an injective morphism from $\Gamma$ to $\Homeo^+(\Sph^1)$. In particular, this is the first step to show that a subgroup of diffeomorphisms of the plane which preserves orientation and which has a bounded orbit is circularly orderable. To establish more properties of the group $\Gamma$, we carry here out a program proposed by Calegari in \cite{blog-Calegari}.

\subsection{The ray graph}
A central object in the study of mapping class groups of finite type surfaces is the \textit{curve complex}. This is a simplicial complex associated to each surface, whose simplexes are the sets of isotopy classes of essential simple closed curves on the surface which have disjoint representatives. The Gromov-hyperbolicity of this complex, established by Howard Masur and Yair Minsky (see \cite{Masur-Minsky}), is a strong tool to study these groups. In the case of the group $\Gamma$ that we care about, the curve complex of the plane minus a Cantor set is not very interesting from a large scale point of view: it has, indeed, diameter $2$. Danny Calegari suggested to replace this complex by the \emph{ray graph}, and defined it in the following way (see Figure \ref{figu:graphe-rayons} for examples of rays):

\begin{definition} [Calegari \cite{blog-Calegari}]
The \emph{ray graph} is the graph whose vertex set is the set of isotopy classes of proper rays, with interior in the complement of the Cantor set $K$, from a point in $K$ to infinity, and whose edges (of length $1$) are the pairs of such rays that can be realized disjointly.
\end{definition}

\begin{figure}[!h]
\labellist
\pinlabel $\alpha$ at 95 190
\pinlabel $\beta$ at 148 190
\pinlabel $\gamma$ at 202 190
\pinlabel $\alpha$ at 392 137
\pinlabel $\gamma$ at 458 136
\pinlabel $\beta$ at 526 134
\pinlabel $\infty$ at 162 270
\endlabellist
\centering
\includegraphics[scale=0.40]{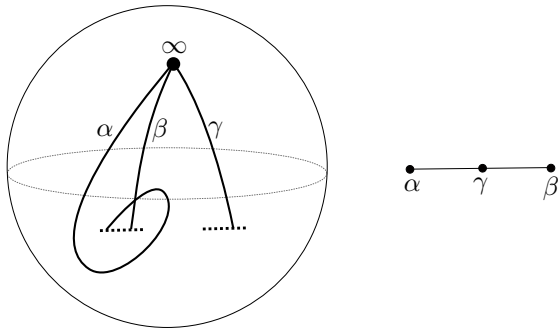}
\caption{Example of three rays on the sphere, and the associated subgraph of the ray graph: $d(\alpha,\beta)=2$ and $d(\alpha,\gamma)=d(\beta,\gamma)=1$.}
\label{figu:graphe-rayons}
\end{figure}

We prove here the following results:

\begin{theo*}[\ref{theo-infini}]
The ray graph has infinite diameter.
\end{theo*}

\begin{theo*}[\ref{theo-hyperbolique}]
The ray graph is Gromov-hyperbolic.
\end{theo*}

\begin{theo*}[\ref{theo-halphak}] 
There exists an element $h \in \Gamma$ which acts by translation on a geodesic axis of the ray graph.
\end{theo*}

These results allow us to see $\Gamma$ as a group acting non trivially on a Gromov-hyperbolic space. We then use this action to construct non trivial quasimorphisms on $\Gamma$.

\subsection{Quasimorphisms and bounded cohomology}

A \emph{quasimorphism} on a group $G$ is a map $q:G\rightarrow \R$ such that there exists a constant $D_q$, called default of $q$, which satisfies the following inequality for all $a,b \in G$:
$$|q(ab)-q(a)-q(b)|\leq D_q.$$
The first examples of quasimorphisms are morphisms and bounded functions. These examples are called \emph{trivial quasimorphisms}. We say that a quasimorphism $q$ is non trivial if the quasimorphism
$\tilde q$ defined by $\tilde q(a) = \lim_{n\rightarrow \infty} {q(a^n)\over n}$ for all $a\in G$ is \emph{not} a morphism.

The \emph{space of classes of non trivial quasimorphisms} on a group $G$, that we will denote by $\tilde Q(G)$, is defined as the quotient of the space of quasimorphisms on $G$ by the direct sum of the subspace of bounded function and the subspace of real morphisms on $G$. Note that the existence of non trivial quasimorphisms on $G$ is equivalent to the existence of non zero elements in $\tilde Q(G)$.

The space $\tilde Q(G)$ coincides with the kernel of the natural morphism which sends the second group of bounded cohomology $H^2_b(G;\R)$ of $G$ in the second group of cohomology $H^2(G;\R)$ of $G$ (see for example Barge \& Ghys \cite{Barge-Ghys} and Ghys \cite{Ghys-Groups} for more details on bounded cohomology of groups). The study of this space $\tilde Q(G)$ gives information on the group $G$: for example, we know that it is trivial when $G$ is amenable (see Gromov \cite{Gromov}), or when $G$ is a cocompact irreducible high rank lattice (see Burger \& Monod~\cite{Burger-Monod}).

In \cite{Bestvina-Fujiwara}, Mladen Bestvina and Koji Fujiwara proved that $\tilde Q(G)$ has infinite dimension when $G$ is the mapping class group of a finite type surface. This result has many consequences, and in particular the authors proved that if $H$ is an irreducible lattice in a connected semi-simple Lie group with no compact factors, with finite center, and of rank greater than $1$, then every morphism from $H$ to the mapping class group of a finite type surface has finite image.

These results, as well as potential applications in dynamics, motivate the research of non trivial quasimorphims on $\MCG(\R^2-\Cantor)$ proposed by Calegari \cite{blog-Calegari}. 
We show here the following result:
\begin{theo*}[\ref{dim infinie}] The space $\tilde Q(\Gamma)$ of classes of non trivial quasimorphisms on $\Gamma$ has infinite dimension.
\end{theo*}

In particular, this implies that the \emph{stable commutator length} is unbounded on $\Gamma$.

\subsection{Stable commutator length}

When $G$ is a group, we denote by $[G,G]$ its derived subgroup, i.e. the subgroup of $G$ generated by commutators. For all $a \in [G,G]$, we denote by $cl(a)$ the \emph{commutator length} of $a$, i.e. the smallest number of commutators whose product is equal to $a$. We defined the \emph{stable commutator length} of $a$ by:
$$\scl(a):=\lim_{n\rightarrow +\infty} {cl(a^n)\over n}.$$

In particular, this quantity is invariant by conjugation (see Calegari \cite{SCL} for more details on the stable commutator length). The study of this quantity is related to non trivial quasimorphisms by a duality theorem: Christophe Bavard proved in \cite{CB} that the space of classes of non trivial quasimorphisms on a group $G$ is trivial if and only if all the elements of $[G,G]$ have vanishing $\scl$.

In the case of $\Gamma$ that we are interested in, Danny Calegari showed in \cite{blog-Calegari} that if $g \in \Gamma$ has a bounded orbit on the ray graph, then $\scl(g)=0$. This property distinguishes the action of $\Gamma$ on the ray graph from the action of mapping class group of finite type surfaces on curve complexes: indeed, Endo \& Kotschick \cite{Endo-Kotschick} and Korkmaz \cite{Korkmaz-scl} proved that Dehn twists (which have bounded orbits on curve complexes) have positive $\scl$.

In the finite type setting, we now know how to characterize precisely the elements with vanishing $\scl$ in terms of the Nielsen-Thurston classification (see Bestvina, Bromberg \& Fujiwara \cite{Bestvina-Bromberg-Fujiwara}). For $\Gamma$, one could ask whether the converse of Calegari's proposition is true: do every elements of $\Gamma$ with vanishing $\scl$ have a bounded orbit on the ray graph? We exhibit here a loxodromic element of $\Gamma$ with vanishing $\scl$ (Proposition~\ref{ex-scl_nulle-hyp}), proving that a characterization of the elements of $\Gamma$ having vanishing $\scl$ would be more refined that the classification between elements having bounded or unbounded orbits.

\subsection{Ideas of proofs }

\subsubsection{Infinite diameter}
In Section \ref{section1}, we construct a sequence of rays $(\alpha_k)_k$ which is unbounded in the ray graph, proving that the ray graph has infinite diameter.

\begin{figure}[!h]
\labellist
\pinlabel $\infty$ at 49 458
\pinlabel $a_1$ at 176 416
\pinlabel $a_2$ at 212 227
\pinlabel $a_3$ at 223 24
\endlabellist
\centering
\includegraphics[scale=0.45]{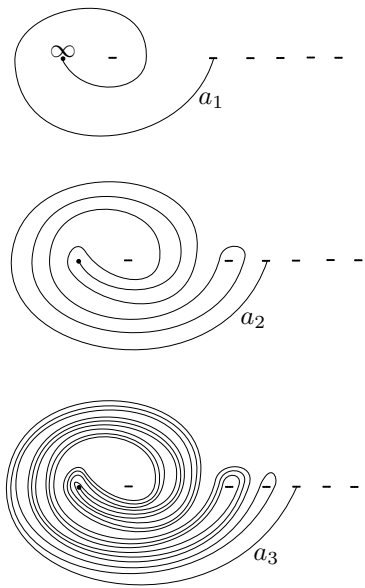}
\caption{Construction of $a_2$ from $a_1$, and of $a_3$ from $a_2$.}
\label{figu:tube}
\end{figure}
  
This sequence is constructed by induction from the following idea: if we consider a representative $a_1$ of some ray and an arc $a_2$ which forms a "tube" in a small neighborhood around $a_1$ (as in figure \ref{figu:tube}), then any arc which is disjoint from $a_2$ and which represents a ray has to starts at infinity and to stop on a point of the Cantor set, without crossing $a_2$. Such an arc has to "follow $a_1$" before it could possibly escape the tube drawn by $a_2$ and reach a point of the Cantor set.
  
Now if $a_3$ is an arc representing a ray and which draws a tube in a small neighborhood of $a_2$ (see figure \ref{figu:tube}), the same phenomenon is true: any arc disjoint from $a_3$ has to "follow $a_2$" before it could possibly escape the tube drawn by $a_3$ and reach a point of the Cantor set.

It follows from these observations that every ray at distance $1$ from the ray represented by $a_3$ has to begin like $a_2$, which forces every ray at distance $2$ from $a_2$ to begins like $a_1$: if for example $\beta$ is the ray represented by an arc which joins infinity to the end point of $a_1$ and stays in the north hemisphere, then the distance between $\beta$ and $a_3$ in the ray graph is at least $3$. Indeed, every arc which begins like $a_2$ or $a_1$ is not homotopically disjoint from $\beta$, thus all the representatives of rays at distance $1$ or $2$ from the ray represented by $a_3$ intersect any arc homotopic to $\beta$.

We then choose $a_4$ which draws a tube around $a_3$: every ray at distance $1$ from the ray represented by $a_4$ begins like $a_3$; this implies that every ray at distance $2$ from $a_4$ begins like $a_2$; this implies that every ray at distance $3$ from $a_4$ begins like $a_1$; and this implies that the distance between the ray represented by $a_4$ and $\beta$ is at least $4$.

We can keep going by choosing $a_5$ which draws a tube around $a_4$, etc. For every $k$, we get a ray $\alpha_k$ represented by $a_k$, and such that every representative of a ray which is at distance smaller than $k$ from $\alpha_k$ begins like $a_1$, and thus intersects $\beta$.

To make all this discussion rigorous, we define in Section~\ref{section1} a coding for some rays, and the sequence $(\alpha_k)_{k\in \N}$ of the rays which draw the needed "tubes". Using the coding, we show that this sequence is unbounded in the ray graph (Theorem~\ref{theo-infini}), and that it defined a geodesic half-axis in this graph (Proposition \ref{alpha_k geod}).

\subsubsection{Hyperbolicity}

In Section \ref{section2}, we prove that the ray graph is Gromov-hyperbolic (Theorem \ref{theo-hyperbolique}). We define an other graph $X_\infty$ whose vertices are isotopy classes of simple loops on $\Sph^2 - K$, based on infinity, and whose edges are pairs of such loops having disjoint representatives. We show that this graph $X_\infty$ is Gromov-hyperbolic by adapting the proof of the uniform hyperbolicity of arc complexes with unicorn paths, given by Sebastian Hensel, Piotr \mbox{Przytycki} and Richard Webb in \cite{HPW}.

We then prove that the graph $X_\infty$ is quasi-isometric to the ray graph, which establishes the Gromov-hyperbolicity of the latter. To this end, we define a map between the ray graph and $X_\infty$, which sends every ray $x \in X_r$ to a loop $\hat x$ of $X_\infty$ such that $x$ and $\hat x$ have disjoint representatives. We prove that this map is a quasi-isometry.

\subsubsection{Loxodromic element}

In Section \ref{section-qm}, we use again the sequence of rays $(\alpha_k)_k$ defined in Section \ref{section1}, which is a geodesic half-axis in the ray graph. We exhibit an element $h \in \Gamma$ which acts by translation on this axis (Theorem \ref{theo-halphak}). The element $h$ can be represented by the braid of figure \ref{figu-h}. The dots represent the Cantor set $K$, and each string carries all the dots of the corresponding subset of the Cantor set. We show that for all $k\in \N$, $h(\alpha_k)=\alpha_{k+1}$.

\begin{figure}[!h]
\labellist
\small\hair 2pt
\pinlabel $\infty$ at 211 204
\pinlabel $h$ at 465 110
\endlabellist
\centering
\vspace{0.2cm}
\includegraphics[scale=0.4]{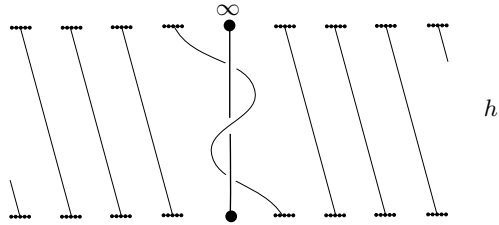}
\caption{Representation of the element $h \in \Gamma$.}
\label{figu-h}
\end{figure}

\subsubsection{Quasimorphisms}
We then construct non trivial quasimorphisms on $\Gamma$. In \cite{Fujiwara}, Koji Fujiwara define counting quasimorphisms on groups acting by isometries on Gromov-hyperbolic spaces, generalizing the construction of Brooks \cite{Brooks} on free groups.
Mladen Bestvina and Koji Fujiwara use this construction to prove that the spaces of classes of non trivial quasimorphisms on mapping class groups of finite type surfaces have infinite dimension (see \cite{Bestvina-Fujiwara}). The Gromov-hyperbolic space that they consider is the curve complex of the surface, and they prove that the action of the mapping class group on this complex is \emph{weakly properly discontinuous}. In particular, this property guarantees the non triviality of some quasimorphisms given by Fujiwara's construction.

Now that we know that $\Gamma$ also acts on a Gromov-hyperbolic graph (the ray graph), Fujiwara's construction gives us quasimorphisms on $\Gamma$. We then need to prove that some of them are non trivial. Unfortunately, the action of $\Gamma$ on the ray graph is not \emph{weakly properly discontinuous} (see the beginning of Section \ref{rq-WPD}). However, we can count the \emph{"number of positive intersections"}, and use it to prove that the axis $(\alpha_k)_k$ is \emph{non reversible} (Proposition \ref{copies}). This property generalizes the fact for $h$ to be non conjugated to its inverse. More precisely, we prove that for every sufficiently long oriented segment $w$ along this axis $(\alpha_k)_k$, if some element of $\Gamma$ maps $w$ in a neighborhood of $(\alpha_k)_k$, then the image of $w$ by this element is oriented in the same direction than $w$. This property of the axis $(\alpha_k)_k$, as well as the action of $h$ on this axis, help us to construct an explicit non trivial quasimorphism on $\Gamma$ (Proposition~\ref{prop-qm}).

We then use again the element $h \in \Gamma$, as well as a conjugate of its inverse, to show using another theorem of Bestvina and Fujiwara \cite{Bestvina-Fujiwara} and the Proposition~\ref{copies} of non-reversibility, that the space $\tilde Q(\Gamma)$ of classes of non trivial quasimorphisms on $\Gamma$ has infinite dimension (Theorem \ref{dim infinie}).

\subsubsection{Scl}
Finally, we give an example of an element of $\Gamma$ with vanishing $\scl$ and a loxodromic action on the ray graph.

\subsection{Acknowledgments}
I thank my PhD advisor, Fr\'{e}d\'{e}ric Le Roux, for his great availability, his numerous advices, and his careful reviews of the different (french) versions of this text. I thank Danny Calegari for his interest to this work, and for suggesting to add an example of a loxodromic element with vanishing $\scl$, in addition to the questions asked on his blog... I also thank Nicolas Bergeron for his explanations around hyperbolic surfaces.

This work was supported by grants from Région Ile-de-France\footnote{And the author benefits from the support of the French government “Investissements d’Avenir”, program ANR-11-LABX-0020-01, while translating this text in English.}.

\section{First study of the ray graph: infinite diameter and geodesic half-axis}\label{section1}

We show here that the ray graph has infinite diameter. We construct a sequence of rays $(\alpha_n)_{n\geq0}$  and we show that this sequence is unbounded in the ray graph. 
More precisely, we code some rays by sequences of segments, to be able to handle them more easily during the proofs. We define with this coding the sequence of rays $(\alpha_n)_n$ that we care about. Finally, we show that this sequence is unbounded in the ray graph, and that it defines a geodesic half-axis in the graph. The results of this section will be used in Section \ref{section-qm}.

\subsection{Preliminaries}
In the rest of the paper, we will use the following notations, propositions and vocabulary.

\subsubsection*{Cantor set $K$}
We denote by $K$ a Cantor set embedded in $\Sph^2$, and we choose a point of $\Sph^2 - K$ that we denote by $\infty$. We identify $\R^2 -K$ and $\Sph^2-(K\cup \{\infty\})$. It is known that if $K'$ is another Cantor set embedded in $\Sph^2$, and if $\infty'$ is a point of $\Sph^2 - K'$, then there exists a homeomorphism of $\Sph^2$ which maps $K'$ on $K$ and $\infty'$ on $\infty$ (see for example the appendix $A$ of Béguin, \mbox{Crovisier~\&~Le~Roux \cite{Beguin-Crovisier-FLR_Cantor}).}

\subsubsection*{Arcs, homotopies and isotopies}

Let ${a}:[0,1] \rightarrow \Sph^2$ be a continuous map such that $\{ {a}(0)\}$ and $\{{a}(1)\}$ are included in $K\cup \{\infty\}$, and such that ${a}(]0,1[)$ is included in $\Sph^2-(K\cup \{\infty\})$. We call \emph{arc} this map ${a}$, and to simplify we call $a$ again the image of $]0,1[$ by ${a}$. If moreover the map ${a}$ is injective, we say that ${a}$ is a \emph{simple arc} of $\Sph^2-(K \cup \{\infty\})$.\\

We say that two arcs ${a}$ and ${b}$ of $\Sph^2-(K \cup \{\infty\})$ are \emph{homotopic} if there exists a continuous map $H:[0,1]\times [0,1] \rightarrow \Sph^2$ such that:
\begin{itemize}
\item $H(0,\cdot)={a}(\cdot)$ and $H(1,\cdot)={b}(\cdot)$;
\item $H(\cdot,0)$ and $H(\cdot,1)$ are constant (the endpoints are fixed);
\item $H(t,s) \in \Sph^2-(K \cup \{\infty\})$ for all $(t,s) \in [0,1]\times ]0,1[$.
\end{itemize}

If $a$ and $b$ are simple, homotopic, and if moreover there exists a homotopy $H$ such that for all $t\in[0,1]$, $H(t,\cdot)$ is a simple arc, then we say that $a$ and $b$ are \emph{isotopic}. David Epstein proved in \cite{Epstein} that on a surface, two homotopic arcs are necessarily isotopic. In this paper, we will use indifferently homotopy and isotopy on surfaces.

We say that two isotopies classes of arcs $\alpha$ and $\beta$ are \emph{homotopically disjoint} if there exist two representatives $a$ of $\alpha$ and $b$ of $\beta$ so that $a(]0,1[)$ and $b(]0,1[)$ are disjoint. We say that two arcs $a$ and $b$ are \emph{homotopically disjoint} if they represent two homotopically disjoint isotopy classes. A \emph{bigon} between two arcs $a$ and $b$ is a connected component of the complement of $a\cup b$ in $\Sph^2-(K\cup \{\infty\})$, which is homeomorphic to a disk, and whose boundary of its closure is the union of a subarc of $a$ and a subarc of $b$. We say that two proper arcs $a$ and $b$ are in \emph{minimal position} if all their intersections are transverse, and if there is no bigon between $a$ and $b$.

\subsubsection*{Ray graph}
\begin{definition}
A  \emph{ray} is an isotopy class $\alpha$ of simple arcs with endpoints $\alpha(0)=\infty$ and $\alpha(1)\in K$. We call \emph{Cantor-endpoint of $\alpha$} the point $\{\alpha(1) \}$.
\end{definition}

\begin{definition} [Calegari \cite{blog-Calegari}]
The  \emph{ray graph}, denoted by $X_r$, is the graph defined as follow~:
\begin{itemize}
\item The set of vertices is the set of rays previously defined;
\item Two vertices are joined by an edge if and only if they are homotopically disjoint.
\end{itemize}
\end{definition}

\begin{remark} The ray graph is connected. Observe that if two rays $\alpha$ and $\beta$ have infinitely many intersections (up to isotopy), then they necessarily have the same Cantor-endpoint. Thus there exists a ray $\alpha'$ which is disjoint from $\alpha$ and intersects $\beta$ finitely many times (take any ray disjoint from $\alpha'$ and with a distinct Cantor-endpoint). We can then adapt the classical proof of the connectedness of the curve complex, given for example in Farb \& Margalit~\cite{Farb-Margalit}, Theorem $4.3$ page $97$, to find a path between $\alpha'$ and $\beta$.
\end{remark}

\subsubsection*{Preliminaries on isotopy classes of curves}
We will use the following results, adapted from Casson \& Bleiler~\cite{Casson-Bleiler}, Handel \cite{Handel} and Matsumoto \cite{Matsumoto}. We equip $\Sph^2 - (K\cup \{ \infty\})$ with a complete hyperbolic metric of the first kind. Its universal cover is the hyperbolic plane $\mathbb{H}^2$.

\begin{prop} \label{prop 3.5}
Let $\mathcal A$ and $\mathcal B$ be two locally finite families of simple arcs of $\Sph^2-(K\cup \{\infty\})$ such that all the elements of $\mathcal A$ (respectively $\mathcal B$) are mutually homotopically disjoint. Assume that for all $a \in \mathcal A$ and $b \in \mathcal B$, $a$ and $b$ are in minimal position.

Then there exists a homeomorphism $h$ which is isotopic to the identity by an isotopy which fixes $K\cup \{ \infty \}$, and such that for all $a \in \mathcal A$ and $b \in \mathcal B$, $h(a)$ and $h(b)$ are geodesic.
\end{prop}

\begin{prop} \label{prop-hyp}
Let ${a}$ and ${b}$ be two arcs of $\Sph^2 - (K \cup \{\infty\})$. If $\tilde {a}$ is a lift of $a$ in the universal cover, then there exist two points $p^-$ and $p^+$ on the boundary $\partial \mathbb{H}^2$ of the universal cover $\mathbb{H}^2$ such that $\tilde a(t)$ goes to $p^-$, respectively $p^+$, when $t$ goes to $0$, respectively $1$. We call endpoints of $\tilde a$ these two points. If $\tilde {a}$ and $\tilde {b}$ are two lifts of ${a}$ and ${b}$ respectively, which have the same endpoints in the boundary of the universal cover, then $a$ and $b$ are isotopic in $\Sph^2 - (K \cup \{\infty\})$.
\end{prop}

\subsection{Coding of some rays}\label{partie codage}

\subsubsection*{Equator}

Using Proposition \ref{prop 3.5}, we choose a topological circle $\E$ of $\Sph^2$ which contains $K \cup \{ \infty\}$ and so that all the open segments of $\E - (K\cup \{\infty\})$ are geodesics for the previous metric on $\Sph^2 - (K\cup \{ \infty\})$. We call \emph{equator} this circle. We choose an orientation on the equator. We call \emph{northern hemisphere} the topological disk on the left of the equator, and \emph{southern hemisphere} the topological disk on its right.

\begin{figure}[!h]
\labellist
\small\hair 2pt
\pinlabel $\infty$ at -13 122
\pinlabel $s_{-1}$ at 26 174
\pinlabel $s_0$ at 42 80
\pinlabel ${p}_{-1}$ at -5 215
\pinlabel ${p}_0$ at 57 24
\pinlabel $s_1$ at 104 24
\pinlabel $s_2$ at 187 16
\pinlabel $s_3$ at 263 59
\pinlabel $s_4$ at 298 140
\pinlabel ${p}_1$ at 150 -12
\pinlabel ${p}_2$ at 251 11
\pinlabel ${p}_3$ at 316 87
\pinlabel ${p}_4$ at 332 186
\pinlabel $p$ at 199 327
\pinlabel \textit{Northern hemisphere} at 137 248
\pinlabel \textit{Southern hemisphere} at -77 248
\endlabellist
\centering
\includegraphics[scale=0.5]{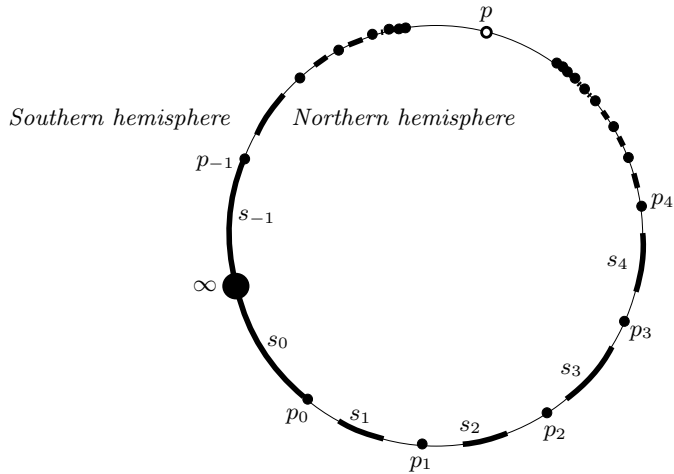}
\vspace{0.5cm}
\caption{Choice of an equator, a point $p$, a sequence of points in $K$, and a set of segments.}
\label{figu:segments}
\end{figure}

\subsubsection*{Choice of segments of $\E$}

As in Figure \ref{figu:segments}, we choose a point $p$ in $\E - \{\infty\}$ such that the two connected components of $\E - \{\infty,p\}$ both contains points of $K$. We then choose a sequence $(p_n)_{n\in \N}$ of points of $K$ in the connected component of $\E - \{\infty,p\}$ which is on the right of $\infty$, in such a way so that ${p}_0$ is the first point of $K$ on the right of $\infty$ on $\E$, and $p_{n+1}$ is on the right of $p_n$ for all $n\in \N$. We choose a sequence $(p_n)_{n<0}$ in the same way in the connected component of $\E - \{\infty,p\}$ which is on the left of $\infty$, such that ${p}_{-1}$ is the first point on the left of $\infty$, and such that $p_{n-1}$ is on the left of $p_n$ for all $n <0$. We denote by $s_0$ the connected component of $\E - (K\cup \{\infty\})$ between $\infty$ and $p_0$, and by $s_{-1}$ the one between $\infty $ and $p_{-1}$. For all $n>0$, we choose a connected component $s_n$ of $\E-K$ between $p_{n-1}$ and $p_n$, and for all $n<-1$, a connected component $s_n$ of $\E - K$ between $p_{n}$ and $p_{n+1}$.

We denote by $S$ the set of topological segments $\{s_n\}_{n\in \Z}$, and by $\textbf{S}$ their union $\bigcup_{n\in \Z} s_n$.

\subsubsection*{Associated sequence}

\begin{figure}[!h]
\labellist
\small\hair 2pt
\pinlabel $\infty$ at 51 182
\pinlabel ${p}_0$ at 116 84
\pinlabel ${p}_1$ at 213 50
\pinlabel ${p}_2$ at 310 72
\pinlabel ${p}_3$ at 384 157
\pinlabel ${p}_4$ at 394 247
\pinlabel $\gamma$ at 3 50
\pinlabel \textit{North} at 164 332
\pinlabel \textit{South} at 64 332
\endlabellist
\centering
\includegraphics[scale=0.5]{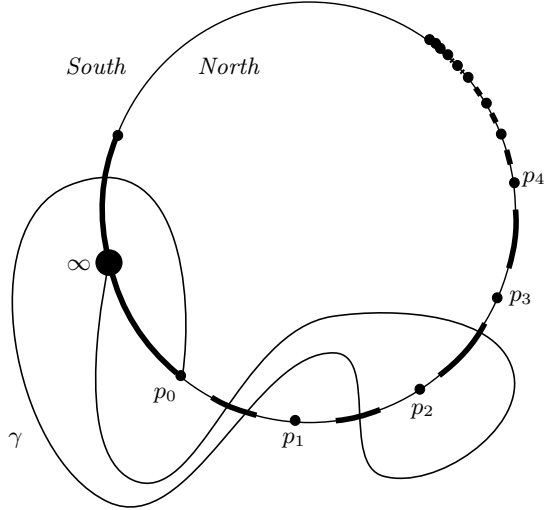}
\vspace{0.3cm}
\caption{Example of a ray $\gamma \in X_S$: in this example, the Cantor-endpoint is ${p}_0$, the complete associated sequence of segments is $u(\gamma)=s_1s_3s_2s_1s_{-1}({p}_0)$, and we have $\mathring u(\gamma)=s_1s_3s_2s_1s_{-1}$.}
\label{exemple-suite}
\end{figure}

If $\alpha$ is an isotopy class of arcs of $\Sph^2-(K \cup \{\infty\})$, we denote by $\alpha_\#$ the unique geodesic representative arc of $\alpha$ in $\Sph^2- (K \cup \{\infty\})$. We denote by $X'_S$ the set of isotopy classes of arcs $\alpha$ in $\Sph^2-(K \cup \{\infty\})$ between infinity and a point of the Cantor set (possibly with self-intersection) such that:
\begin{enumerate}
\item $\E \cap \alpha_\# \subset \textbf{S}$ ;
\item The connected component $\alpha_\#  - \E$  which starts at $\infty$ is included in the southern hemisphere ;
\item $\E \cap \alpha_\#$ is a finite set.
\end{enumerate}

We denote by $X_S$ the subset of $X'_S$ which contains the isotopy classes of \emph{simple} arcs (i.e. the set of rays which satisfy the three previous conditions).

Let $\alpha \in X'_S$. We can associate to $\alpha$ a sequence of segments in the following way: we follow $\alpha_\#$ from $\infty$ and to its Cantor-endpoint, and we denote by $u_1$ the first segment of $S$ intersected by $\alpha_\#$, $u_2$ the second, ..., and $u_k$ the k-th for all $k$, until we get to the Cantor-endpoint.

We denote by $\mathring u(\alpha)$ this (finite) sequence of segments, and by $u(\alpha)$ the sequence $\mathring u(\alpha)$ together with the Cantor-endpoint of $\alpha$. We call it \emph{complete sequence associated to $\alpha$} (see Figure \ref{exemple-suite} for an example). Because the geodesic $\alpha_\#$ in the isotopy class $\alpha$ is unique, the sequence associated to $\alpha$ is well defined. More generally, we will call \emph{complete sequence} any finite sequence of segments together with a point of $K$, such that the sequence of segments does not begin with $s_{-1}$ nor $s_0$, and does not contain any repetition of segments (i.e. $u_i \neq u_{i+1}$ for all $i$), in order to avoid bigons.

\begin{lemma}
To each complete sequence corresponds a unique isotopy class of arcs in $X'_S$ (possibly with self-intersections) between infinity and a point of $K$. In particular, if two rays of $X_S$ have the same complete associated sequence, then they are equal.
\end{lemma}

\begin{proof} Let $\alpha$ and $\beta$ be two arcs with the same complete associated sequence, let say $u_0...u_n({p}_j)$. On the boundary of the universal cover, we choose a "lift" $\tilde \infty$ of $\infty$: we can see this point $\tilde \infty$ as the limit on the boundary of any chosen lift of $\alpha$. We then lift $\beta$ from that point $\tilde \infty$. The universal cover is tessellated by half fundamental domains corresponding to the lifts of the hemispheres: the boundary of any of these half fundamental domains is a lift of the equator. We start to lift $\alpha$ and $\beta$ from $\tilde \infty$ in a same half fundamental domain $F_0$ (which is a lift of the southern hemisphere). We define $(F_i)_{0\leq i\leq n}$ as the sequence of the alternative lifts of the northern and southern hemispheres that are crossed by $\tilde \alpha_\#$. Observe that $(F_i)_i$ is determined by the coding: we go out $F_0$ to arrive in a lift $F_1$ of the northern hemisphere by crossing the only lift of $u_0$ which is in the closure of $F_0$. We keep going this way until we reach the half-domain $F_n$, which has only one lift $\tilde {p}_j$ of ${p}_j$ in its boundary. Thus the two lifts $\tilde \alpha$ and $\tilde \beta$ of $\alpha$ and $\beta$ have the same endpoints. It follows that $\alpha$ and $\beta$ are isotopic in $\Sph^2 - (K \cup \{\infty\})$ (according to Proposition \ref{prop-hyp}). \end{proof}

From now on, we won't make any distinction between an isotopy class of arcs in $X'_S$ and its complete associated sequence.

\subsection{A specific sequence of rays}
We construct here a specific sequence of rays, $(\alpha_k)_{k\in \N}$, whose properties will be useful later. \\

If $u=u_0u_1...u_n ({p}_j)$ is a complete sequence of segments, recall that we denote by \mbox{$\mathring u = u_0 u_1 ...u_n$} the sequence of segments without the Cantor-endpoint. The inverse of this finite sequence of segments will be denoted by \mbox{$\mathring u^{-1}:=u_n...u_1 u_0$}.

\begin{figure}[h!]
\labellist
\small\hair 2pt
\pinlabel $\infty$ at 60 195
\pinlabel $\infty$ at 297 251
\pinlabel $s_{-1}$ at -11 187
\pinlabel $s_0$ at 123 187
\pinlabel $s_{-1}$ at 227 245
\pinlabel $s_0$ at 360 245
\pinlabel $s_{-1}$ at 363 171
\pinlabel $s_1$ at 123 133
\pinlabel $s_{-1}$ at 125 80
\pinlabel $s_1$ at 123 27
\pinlabel $s_{-1}$ at 363 65
\pinlabel $s_1$ at 360 208
\pinlabel $s_{2}$ at 25 34
\pinlabel $s_1$ at 360 106
\pinlabel $s_{k+1}$ at 260 34
\pinlabel $s_k$ at 360 26

\pinlabel ${p}_1$ at 58 18
\pinlabel ${p}_2$ at 3 18
\pinlabel ${p}_k$ at 296 17
\pinlabel ${p}_{k+1}$ at 240 17

\pinlabel \textit{North} at -15 214
\pinlabel \textit{South} at -15 163
\pinlabel \textit{North} at -15 109
\pinlabel \textit{South} at -15 58
\pinlabel \textit{North} at -15 7

\pinlabel \textit{South} at 230 226
\pinlabel \textit{North} at 230 189

\pinlabel \textit{South} at 230 46

\pinlabel $\alpha_1$ at 69 106
\pinlabel $\alpha_2$ at 66 217
\pinlabel $\alpha_k$ at 307 130
\pinlabel $\alpha_{k+1}$ at 304 272
\endlabellist
\centering

\includegraphics[scale=0.8]{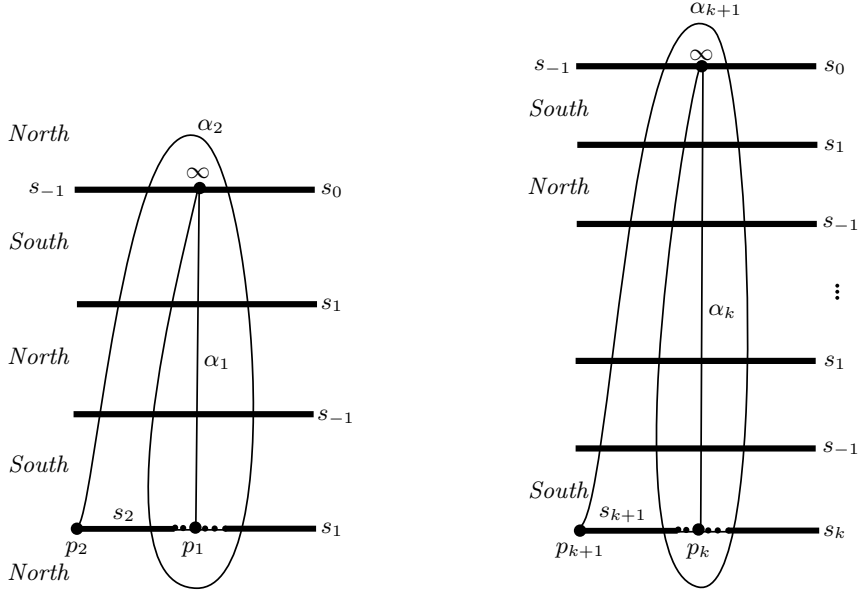}
\caption{Definition of $\alpha_2$ from $\alpha_1$, and of $\alpha_{k+1}$ from $\alpha_k$: representation of the local intersections of these rays with $\E$.}
\label{def-alpha}
\end{figure}

\begin{definition} We define the sequence of rays $(\alpha_k)_{k\geq 0}$ in the following way:
\begin{itemize}
\item $\alpha_0$ is the isotopy class of the segment $s_0$, with endpoints $\infty$ and ${p}_0$;
\item $\alpha_1$ is the ray coded by $s_1 s_{-1} ({p}_1)$ (see Figure \ref{alpha1-2});
\item For all $k\geq 1$, $\alpha_{k+1}$ is the ray defined from $\alpha_k$ as in Fi\-gure~\ref{def-alpha}: we start at $\infty$, we follow $\alpha_{k\#}$ until its Cantor-endpoint ${p}_k$ in a tubular neighborhood of $\alpha_{k\#}$, we turn around this point counterclockwise, crossing the two segments $s_{k+1}$ first, and then $s_k$, we follow $\alpha_{k\#}$ again in a tubular neighborhood, we turn around $\infty$ by crossing $s_0$ first and then $s_{-1}$, we follow $\alpha_{k\#}$ for the last time in a tubular neighborhood until its Cantor-endpoint, and we go to the point ${p}_{k+1}$ of the Cantor set without crossing the equator.
\end{itemize}
\end{definition}

In other words, using the coding, we can define $(\alpha_k)_{k\geq 0}$ with the following complete sequences:
\begin{itemize}
\item $\alpha_0 =s_0({p}_0)$;
\item $\alpha_1 = s_1 s_{-1} ({p}_1)$;
\item $\alpha_{k+1}=\mathring \alpha_k s_{k+1} s_k \mathring \alpha_k^{-1} s_0 s_{-1} \mathring \alpha_k ({p}_{k+1})$ for all $k \geq 1$.
\end{itemize}

\begin{figure}[!h]
\labellist
\small\hair 2pt
\pinlabel $\infty$ at 82 206
\pinlabel ${p}_0$ at 125 102
\pinlabel ${p}_1$ at 225 62
\pinlabel ${p}_2$ at 329 92
\pinlabel ${p}_3$ at 389 164
\pinlabel ${p}_4$ at 404 257
\pinlabel \textit{North} at 154 332
\pinlabel \textit{South} at 64 332
\endlabellist
\centering
\vspace{0.4cm}
\includegraphics[scale=0.43]{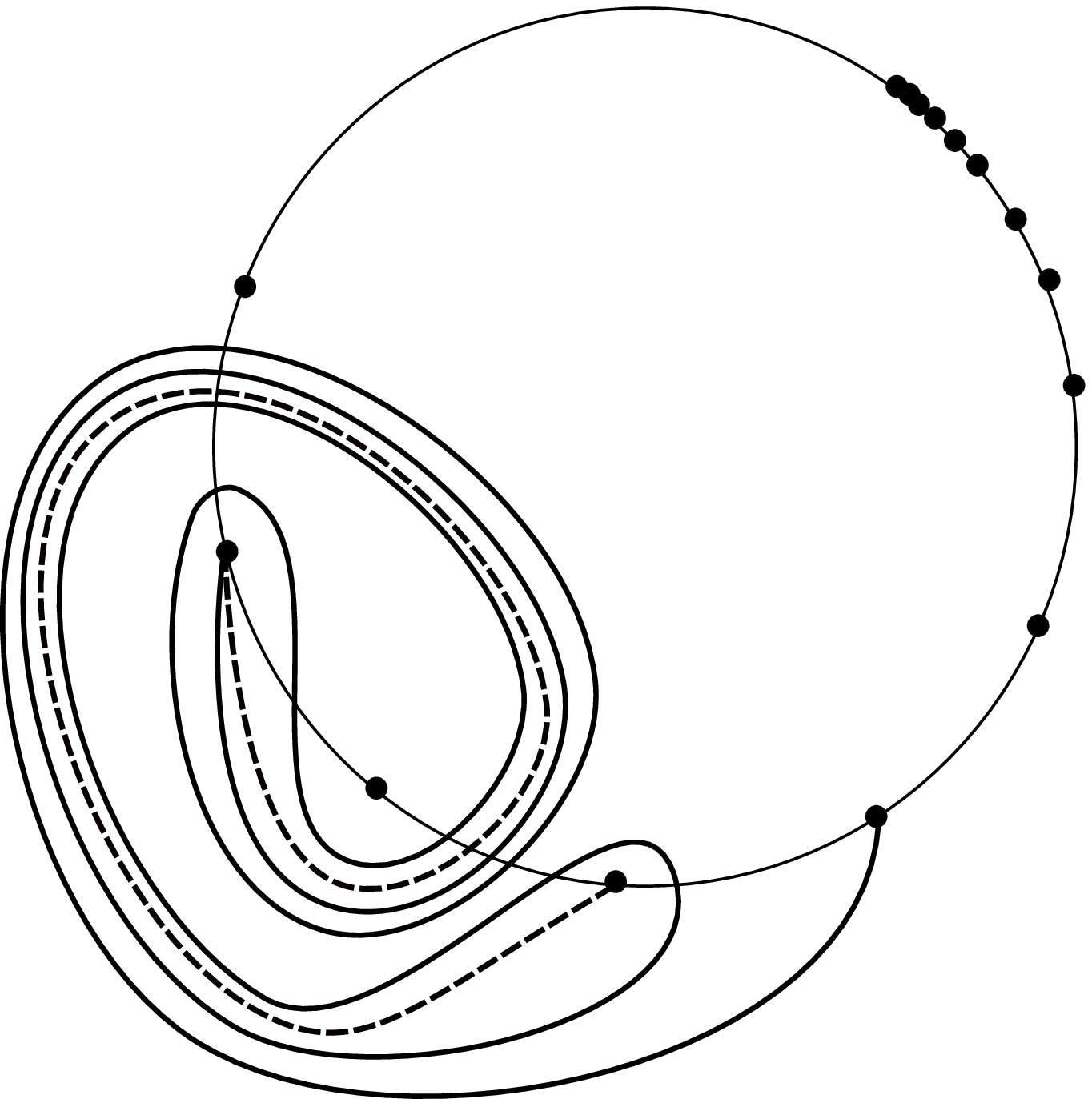}
\caption{On the sphere, representation of $\alpha_1= s_1 s_{-1} ({p}_1)$ (dotted line) and \mbox{$\alpha_2=s_1s_{-1} s_2s_1s_{-1} s_1s_0s_{-1} s_1s_{-1}({p}_2)$} (plain line).}
\label{alpha1-2}
\end{figure}

\begin{remark} If we denote by $\texttt{long}(\alpha_k)$ the number of connected components of \mbox{$\alpha_{k\#} - \E$,} then $\texttt{long}(\alpha_k)$ is odd for all $k \geq 1$. Indeed, we have $\texttt{long}(\alpha_1)=3$ (see Figure \ref{def-alpha}), and by construction $\texttt{long}(\alpha_{k+1}) = 3 \texttt{ long}(\alpha_k)+2$, thus $\texttt{long}(\alpha_{k+1})$ has the same parity than $\texttt{long}(\alpha_k)$. Hence we are sure to be in the configuration of Figure \ref{def-alpha}: the last hemisphere crossed by $\alpha_k$ is always the southern hemisphere, thus ${p}_{k+1}$ is always on the left of ${p}_{k}$ in the local representation that we have chosen (Figure \ref{def-alpha}). In particular, when $\alpha_{k+1}$ turn around $\infty$, this ray crosses first $s_0$ and then $s_{-1}$, to avoid self-intersections.
\end{remark}

\subsection{Infinite diameter and half geodesic axis} \label{def de A}

Let $\beta$ be a ray, and let $\mathring u = u_0 u_1...u_n$ be a sequence of segments. We say that $\beta$ \emph{begins like $\mathring u$} if the first connected component of $\beta_ \# - \E$ is in the southern hemisphere and if the first intersections between $\beta_\#$ and $\E$ are, in this order, the segments $u_0,u_1,...,u_n$. In particular, if $\beta \in X_S$, we say that $u(\beta)$ begins like $\mathring u$.

\begin{definition}
Let $A:X_r \rightarrow \N$ be the map which sends every ray $\gamma \in X_r$ to:
$$A(\gamma):=\max\{i\in \N  \emph{ such that } \gamma  \emph{ begins like } \mathring \alpha_i\}.$$
\end{definition}

\noindent Because $\mathring \alpha_0$ is the empty sequence, $A$ is well defined for all $\gamma \in X_r$. Let us now prove that $A$ is $1$-lipschitz.

\begin{lemma} \label{lemme distance}
Let $\beta$ and $\gamma$ be two rays such that $d(\gamma,\beta)=1$. Then~: $$|A(\gamma) - A(\beta)| \leq 1.$$
\end{lemma}

\begin{proof}
Let $n:=A(\beta)$. We denote by $\beta_\#$ and  $\gamma_\#$ the geodesic representatives of $\beta$ and $\gamma$ (see Figure \ref{prop-alpha}). The arc $\beta_\#$ first follows the curve representing $\mathring \alpha_n$: indeed, it has to cross the same segments, in the same order. There exists a homeomorphism fixing each point of $K$ and $\infty$, preserving $\E$, and sending the beginning of $\beta_\#$, i.e. the component of $\beta_\#$ between $\infty$ and $s_{-1}$, on the beginning of $\alpha_n$, i.e. on the component of $\alpha_n$ between $\infty$ and $s_{-1}$. Because the distance between $\gamma$ and $\beta$ is $1$, $\gamma_\#$ is disjoint from $\beta_\#$ and has to escape the grey area, which does not contain any point of $K$, in order to end on a point of $K$ without intersecting $\beta_\#$, thus without crossing the plain line of Figure \ref{prop-alpha}. It follows that $\gamma_\#$ has to begin by following one of the two dotted arrows, which exactly means that $\gamma$ begins like $\mathring \alpha_{n-1}$. Thus we get $A(\gamma)\geq n-1$. The Lemma follows by symmetry between $\beta$ and $\gamma$. \end{proof}

\begin{figure}[!h]
\labellist
\small\hair 2pt
\pinlabel $\mathring \alpha_n$ at 90 219
\pinlabel $\infty$ at 62 253
\pinlabel $s_{-1}$ at 127 65
\pinlabel $s_{-1}$ at -10 245
\pinlabel $s_0$ at 127 245
\pinlabel $s_1$ at 127 210
\pinlabel ${p}_n$ at 2 15
\pinlabel ${p}_{n-1}$ at 70 15
\endlabellist
\centering
\includegraphics[scale=0.53]{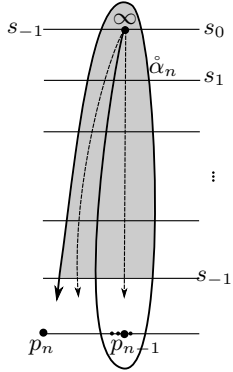}
\caption{Representation of the local intersections of $\mathring \alpha_{n}$ with $\E$. By definition of $(\alpha_k)_k$, there is no point of $K$ in the grey area.}
\label{prop-alpha}
\end{figure}

\begin{corollary} \label{dist}
Let $\beta$ and $\gamma$ be any two rays of $X_r$. We have: $$|A(\beta)-A(\gamma)| \leq d(\beta,\gamma).$$
\end{corollary}

\begin{proof}
We choose a geodesic in the ray graph between $\beta$ and $\gamma$, and by sub-additivity of the absolute value, we deduce this result from Lemma \ref{lemme distance}. \end{proof}

\noindent This inequality allows us to minor some distances, and the following theorem follows: 

\begin{theo}\label{theo-infini}
The diameter of the ray graph is infinite.
\end{theo}

\begin{proof} By definition of $A$, we have $A(\alpha_0)=0$ and $A(\alpha_n)=n$ for all $n\in \N$. According to Corollary \ref{dist}, we have $d(\alpha_0,\alpha_n)\geq n$.
\end{proof}

\begin{prop} \label{alpha_k geod}
The half-axis $(\alpha_k)_{k\in \N}$ is geodesic.
\end{prop}

\begin{proof}
By construction of the sequence $(\alpha_k)_{k\in \N}$, we have $d(\alpha_k,\alpha_{k+1})=1$ for all $k\geq 0$. We also have $d(\alpha_k,\alpha_0)\geq k$ for all $k\geq 0$ (according to Corollary \ref{dist}). Thus for all $k \geq 0$, we have $d(\alpha_k,\alpha_0)=k$.\end{proof}

\section{Gromov-hyperbolicity of the ray graph }\label{section2}

Recall that a metric space $X$ is \emph{geodesic} if between any two points of $X$, there exists at least one geodesic, i.e. a path which minimizes the distance between these two points. We recall the definition of Gromov-hyperbolic metric space. For more details on hyperbolic spaces, see for example Bridson \& Haefliger \cite{Bridson-Haefliger}.

\begin{definition}[Hyperbolic space]
We say that a geodesic metric space $X$ is \emph{Gromov-hyperbolic} (or \emph{hyperbolic}) if there exists a constant $\delta\geq 0$ such that for every geodesic triangle of $X$, each side of the triangle is included in the $\delta$-neighborhood of the two other sides.
\end{definition}

We define a graph $X_\infty$ and we prove that it is Gromov-hyperbolic using the same arguments that those given by Hensel, Przytycki \& Webb in \cite{HPW} to prove the Gromov-hyperbolicity of the arc graph for compact surfaces with boundary. We then use this hyperbolicity to prove that the ray graph is Gromov-hyperbolic.

\subsection{Hyperbolicity of the loop graph $X_\infty$}

\subsubsection*{Graph $X_\infty$ and unicorn paths}

We fix a Cantor set $K$ embedded in $\R^2$ and we compactify $\R^2 $ by adding $\infty$: we get the sphere $\Sph^2$.

A simple arc of $\Sph^2-K$ between infinity and infinity (i.e. a simple loop based at infinity) is said to be \emph{essential} if it does not bound a topological disk, meaning that it splits the sphere into two components which both contain points of $K$.

\begin{definition}
We define the graph $X_\infty$ as follow:
\begin{itemize}
\item Vertices are isotopy classes of simple essential arcs on $\Sph^2 - K$ between $\infty$ and $\infty$, where we identify two arcs if they have the same image but opposite orientation;
\item Two vertices are joined by an edge if and only if they are homotopically disjoint.
\end{itemize}
\end{definition}

\begin{remark} We recall that we denote by $X_r$ the ray graph. The graphs $X_\infty$ and $X_r$ are naturally equipped with a metric where all the edges have length $1$. The group $\Gamma = \MCG(\R^2 - K)$ acts on $X_\infty$ (and on $X_r$) by isometries.
\end{remark}

We adapt here the proof from \cite{HPW} of the hyperbolicity of the arc graph of a compact surface with boundary to prove the hyperbolicity of $X_\infty$.

Let $a$ and $b$ be two simple essential arcs on $\Sph^2 - K$ between $\infty$ and $\infty$ and which are in minimal position. We choose an orientation on each of these arcs and we denote by  $a^+$, $b^+$ the corresponding oriented arcs. Let $\pi \in a\cap b$. Let $a'$ and $b'$ be the oriented subarcs of $a$, respectively $b$, beginning like $a$, respectively like $b$, and having $\pi$ for second endpoint. We denote by $a' \star b'$ the concatenation of these two subarcs; in particular, it is an arc between $\infty$ and $\infty$. Assume that this arc is simple. Because $a$ and $b$ are in minimal position, the arc $a' \star b'$ is essential. Thus it defines an element of $X_\infty$. The arc $a' \star b'$ is said to be \emph{a unicorn arc made from $a^+$ and $b^+$}.

Observe that this arc is uniquely determined by the choice of $\pi \in a\cap b$, and that not all the points of $a \cap b$ define a simple arc. Moreover, $a \cap b$ is a finite set, because $a$ and $b$ have transverse intersections (and $\infty$ is an isolated puncture). Those there are only finitely many unicorn arcs made from $a^+$ and $b^+$.

\begin{claim}
If $\pi$ and $\pi'$ are two points of $a\cap b$ defining the unicorn arcs $a'\star b'$ and $a'' \star b''$, then $a'' \subset a'$ if and only if $b' \subset b''$.
\end{claim}

\begin{definition}[Total order on unicorn arcs]
Let $a^+$ and $b^+$ be two oriented arcs between $\infty$ and $\infty$ on $\Sph^2 -K$, and which are in minimal position. We order the unicorn arcs between $a^+$ and $b^+$ in the following way: \\
$a' \star b' \leq a'' \star b''$ if and only if $a'' \subset a'$ and $b' \subset b''$. 
\end{definition}

This order is total. We denote by $(c_1,...,c_{n-1})$ the ordered set of the unicorn arcs between $a^+$ and $b^+$. In particular, this order corresponds to the order of the intersection points $\pi$ when we follow $b^+$.

We define unicorn paths in $X_\infty$ in the following way:
\begin{definition}[Unicorn paths between oriented arcs] Let $a^+$ and $b^+$ be two oriented essential simple arcs between $\infty$ and $\infty$ on $\Sph^2 -K$, and which are in minimal position.
The sequence of unicorn arcs $P(a^+,b^+)=(a=c_0,c_1,...,c_{n-1},c_n=b)$ is called  \emph{unicorn path between $a^+$ and $b^+$}.
\end{definition}

\begin{claim} Let $a$ and $b$ be two oriented arcs in minimal position and let $(c_0,...,c_n)$ be the unicorn path between these two arcs. Let $a'$ and $b'$ be two arcs in minimal position and such that $a'$, respectively $b'$, is isotopic to $a$, respectively to $b$, and oriented in the same direction. We denote by $(d_0,d_1,...,d_{m-1},d_m)$ the unicorn path between $a'$ and $b'$. Then $n=m$ and $c_k$ is isotopic to $d_k$ for all $k$.\end{claim}
This is a consequence of Proposition \ref{prop 3.5}. \qed

\begin{definition}[Unicorn paths between oriented vertices of $X_\infty$] 
Let $\alpha^+$ and $\beta^+$ be two oriented elements of $X_\infty$. Let $a^+$ and $b^+$ be two oriented representatives of $\alpha^+$ and $\beta^+$ respectively, and which are in minimal position. Let $P(a^+,b^+)=(c_0,...,c_n)$ be the associated unicorn path. For all $1\leq k \leq n$, we denote by $\gamma_k$ the isotopy class of $c_k$. We define the \emph{unicorn path between $\alpha^+$ and $\beta^+$} by: $$P(\alpha^+,\beta^+)=(\gamma_0,\gamma_1,...,\gamma_n).$$
\end{definition}

\begin{claim} Every unicorn path is a path in $X_\infty$. \end{claim}
Indeed, for all \mbox{$0\leq i\leq n-1$}, $c_i$ and $c_{i+1}$ are homotopically disjoint. This is Remark $3.2$ of \cite{HPW}. \qed

\begin{remark} 
\begin{enumerate}
\item If $a\cap b=\emptyset$, then we have $P(a^+,b^+)=(a,b)$.
\item Abusing notations, we will again denote by $P(a^+,b^+)$ the set of the elements of $P(a^+,b^+)$.
\end{enumerate}
\end{remark}

The unicorn arcs only depend on the neighborhood of $a \cup b$: if we consider a closed neighborhood of $a \cup b$ which homotopically equivalent to $a \cup b$, then we can see the unicorn arcs made from $a$ and $b$ as unicorn arcs of the compact surface defined by the closure of this neighborhood. We are then in the exact same setting than in \cite{HPW}. This correspondence allows us to see every unicorn path in $X_\infty$ as a unicorn path in the arc graph of a finite type surface. In particular, Lemmas $3.3$, $3.4$, $3.5$ and $4.3$ of \cite{HPW} are still true for $X_\infty$. As Proposition $4.2$, and Theorem $1.2$ are consequences of these Lemmas, we get the hyperbolicity of $X_\infty$. Note that it seems difficult to deduce the hyperbolicity of $X_\infty$ from the hyperbolicity of the arc graph of only one surface: in the proof of the lemmas, we need to consider different surfaces, which depend on the elements of $X_\infty$ that we are considering. But because the constant of hyperbolicity given by \cite{HPW} does not depend of the surface, this won't be an issue. We adapt the proof of \cite{HPW} in our context. Lemma \ref{slim unicorn}, Corollary \ref{lemme4.3}, Lemma \ref{lemme3.5} and Propositions \ref{prop4.2} and \ref{G_infini hyp} correspond, in this order, to Lemmas $3.3$, $4.3$, $3.5$, Proposition $4.2$ and Theorem $1.2$ of \cite{HPW}\footnote{We choose not to translate the proofs of these Lemmas and Propositions, which are very close to the proofs written in \cite{HPW}, and which have been checked in this particular setting (in french!) in the original paper.}.\\

Note that the proof of \cite{HPW} does not adapt directly to the ray graph $X_r$: indeed, the arc made from two different representatives of rays oriented from infinity to the Cantor-endpoint goes from infinity to infinity, and thus is not in the ray graph. If we change the definition by choosing the unicorn arc as the union of the beginning of $a$ and the end of $b$, then we get an arc whose isotopy class is a ray, but Lemma \ref{slim unicorn} is false. This is why we defined the graph $X_\infty$.

\subsubsection*{Lemmas on unicorn paths in $X_\infty$}

\begin{lemma}[Unicorn triangles are $1$-slim] \label{slim unicorn}
Let $\alpha^+, \beta^+$ and $\delta^+$ be three elements of $X_\infty$, with an orientation. Then for all $\gamma \in P(\alpha^+,\beta^+)$, one of the elements $\gamma^*$ of $P(\alpha^+,\delta^+) \cup P(\delta^+,\beta^+)$ is such that $d(\gamma,\gamma^*)=1$ in $X_\infty$.
\end{lemma}

\begin{proof}
This is Lemma $3.3$ of \cite{HPW}. \end{proof}

\begin{corollary}\label{lemme4.3}
Let $k\in \N$, $m\leq 2^k$ and let $(\xi_0,...,\xi_m)$ be a path in $X_\infty$. We choose an orientation on the $\xi_i$'s. Then $P(\xi_0^+,\xi_m^+)$ is included in the $k$-neighborhood of $(\xi_0,...,\xi_m)$.
\end{corollary}
\begin{proof}

This is Lemma $4.3$ of \cite{HPW}. \end{proof}

\begin{lemma} \label{lemme3.5}
Let $\alpha^+,\beta^+ \in X_\infty$ be oriented and let $P(\alpha^+,\beta^+)=(\gamma_0,...,\gamma_n)$ be the associated unicorn path in $X_\infty$. For all $0\leq i \leq j \leq n$, we consider $P(\gamma_i^+,\gamma_j^+)$, where $\gamma_i^+$, respectively $\gamma_j^+$, have the same orientation than $a^+$, respectively than $b^+$. Then either $P(\gamma_i^+,\gamma_j^+)$ is a subpath of $P(\alpha^+,\beta^+)$, or $j=i+2$ and $d(\gamma_i,\gamma_j)=1$ in~$X_\infty$.
\end{lemma}
\begin{proof}
 This is Lemma $3.5$ of \cite{HPW}.
\end{proof}
\subsubsection*{Hyperbolicity of $X_\infty$}

We can now deduce from the previous lemmas the hyperbolicity of $X_\infty$.

\begin{prop}\label{prop4.2}
Let $\mathcal{G}$ be a geodesic path $X_\infty$ between two vertices $\alpha$ and $\beta$ of $X_\infty$. Then, for any choice of orientation on $\alpha$ and $\beta$, $P(\alpha^+,\beta^+)$ is included in the $6$-neighborhood of $\mathcal{G}$.
\end{prop}

\begin{proof}
This is Proposition $4.2$ of \cite{HPW}.\end{proof}

\begin{corollary}\label{coro++}
Let $\mathcal G$ be a geodesic of $X_\infty$ between two vertices $\alpha$ ans $\beta$. Then for any choice of orientation on $\alpha$ and $\beta$, $\mathcal{G}$ is included in the \mbox{$13$-neighborhood} of~$P(\alpha^+,\beta^+)$.
\end{corollary}

This is a consequence of Proposition \ref{prop4.2} and of the standard following Lemma\footnote{Again, we didn't translate the proof of this Lemma.}:
\begin{lemma} Let $X$ be a geodesic metric space. Let $\mathcal G$ be a geodesic of $X$ between two points $\alpha$ and $\beta$. Let $k$ be a positive integer. If $\mathcal J$ is a path of $X$ between $\alpha$ and $\beta$ which stays in the $k$-neighborhood of $\mathcal G$, then $\mathcal G$ stays in the $(2k+1)$-neighborhood of $\mathcal J$.
\end{lemma}

%
%
\begin{prop} \label{G_infini hyp}
The graph $X_\infty$ is $20$-hyperbolic.
\end{prop}

\begin{proof}
Let $\alpha \beta \gamma$ be a geodesic triangle in $X_\infty$. Let $\zeta$ on the geodesic between $\alpha$ and $\beta$. We choose an orientation on $\alpha$, $\beta$ and $\gamma$. According to Corollary \ref{coro++}, there exists $\xi$ on $P(\alpha^+,\beta^+)$ such that $d(\zeta,\xi) \leq 13$. According to Lemma \ref{slim unicorn}, there exists $\xi^* \in P(\alpha^+,\gamma^+) \cup P(\gamma^+,\beta^+)$ such that $d(\xi,\xi^*)\leq 1$. According to Proposition \ref{prop4.2}, there exists $\zeta^*$ on one of the two geodesic sides of the triangle, either between $\alpha$ and $\gamma$ or between $\gamma$ and $\beta$, such that $d(\xi^*,\zeta^*)\leq 6$. Thus we have $d(\zeta,\zeta^*)\leq 20$. \end{proof}

\subsection{Quasi-isometry between $X_r$ and $X_\infty$} \label{def-q.i.}

We want to deduce the hyperbolicity of the ray graph $X_r$ from the hyperbolicity of $X_\infty$. To this end, we show that the two graphs are quasi-isometric.

\subsubsection*{Reminder on large scale geometry}
We use the following standard definitions and results (see for example Bridson \& Haefliger \cite{Bridson-Haefliger}).

\begin{definition}[Quasi-isometry] Let $X$ and $X'$ be two metric spaces. A map $f : X' \rightarrow X$ is a $(\kappa,\varepsilon)$-quasi-isometric embedding if there exists $\kappa \geq 1$ and $\varepsilon \geq 0$ such that for all $x,y \in X'$:
$${1\over \kappa}d(x,y)-\varepsilon \leq d(f(x),f(y))\leq \kappa d(x,y) +\varepsilon.$$

If moreover there exists $C \geq 0$ such that every element of $X$ is in the $C$-neighborhood of $f(X')$, we say that $f$ is a \emph{$(\kappa,\varepsilon)$-quasi-isometry}. When such a map exists, we say that $X$ and $X'$ are quasi-isometric.
\end{definition}

\begin{definition}[Quasi-geodesic]
A \emph{$(\kappa,\varepsilon)$-quasi-geodesic} of a metric space $X$ is a $(\kappa,\varepsilon)$-quasi-isometric embedding of an interval of $\R$ to $X$. For abbreviation, we call quasi-geodesic any image in $X$ of such an embedding.
\end{definition}

\begin{theo*}[Morse Lemma, see for example Bridson \& Haefliger \cite{Bridson-Haefliger}, Theorem~$1.7$ page $401$] \label{Morse} Let $X$ be a $\delta$-hyperbolic metric space. For all $\kappa, \varepsilon$ positive real numbers, there exists a universal constant $B$ which depends only on $\delta$, $\kappa$ and $\varepsilon$, such that every segment which is $(\kappa,\varepsilon)$-quasi-geodesic is in the $B$-neighborhood of any geodesic between its endpoints.
\end{theo*}

We say that $B$ is the \emph{$(\kappa,\varepsilon)$-Morse constant} of the hyperbolic space $X$.

\begin{theo*}[see for example Bridson \& Haefliger \cite{Bridson-Haefliger},  Theorem $1.9$ page $402$]
\label{bridson}
Let $X$ and $X'$ be two geodesic metric spaces and let $f:X'\rightarrow X$ be a quasi-isometric embedding. If $X$ is Gromov-hyperbolic, then $X'$ is also Gromov-hyperbolic.
\end{theo*}

\subsubsection*{Quasi-isometry between $X_r$ and $X_\infty$}

According to Proposition \ref{G_infini hyp}, we know that $X_\infty$ is a hyperbolic space. To prove that $X_r$ is also Gromov-hyperbolic, we are now looking for a quasi-isometric embedding from $X_r$ to $X_\infty$, which would allow us to conclude using the previous theorem. We will actually prove that the chosen embedding is a quasi-isometry.

We define a map $f:X_r \rightarrow X_\infty$ which sends ${x} \in X_r$ to any $\hat {x} \in X_\infty$ such that ${x}$ and $\hat {x}$ are homotopically disjoint in $\Sph^2-{(K\cup \{\infty\}})$.

\begin{prop}\label{prop quasi-isom}
The map $f$ previously defined is a quasi-isometry.
\end{prop}

\begin{lemma} \label{retour}
Let $\hat {x}, \hat {y} \in X_\infty$ and ${x},{y} \in X_r$ such that ${x}$ (respectively ${y}$) is homotopically disjoint from $\hat {x}$ (respectively from $\hat {y}$).
Then: $$d({x},{y})\leq d(\hat {x},\hat {y})+2.$$
\end{lemma}

\begin{remark} In particular, this lemma implies that for all $x,y \in X_r$, $d(x,y)-2\leq d(f(x),f(y))$.
\end{remark}

\begin{proof}
We denote by $n$ the distance between $\hat {x}$ and $\hat {y}$ in $X_\infty$. Let $(\hat \mu_j)_{0\leq j\leq n}$ be a geodesic in $X_\infty$ between $\hat {x}$ and $\hat {y}$ (in particular, $\hat \mu_0 = \hat {x}$ and $\hat \mu_n=\hat {y}$). We want to construct a path $(\mu_1,...,\mu_{n-1})$ of length $(n-1)$ in $X_r$, and show that $d(x,\mu_1)\leq 2$ and $d(\mu_{n-1},y)\leq 2$. For every $\alpha$ in $X_r$ or $X_\infty$, we denote by $\alpha_\#$ the geodesic representative of $\alpha$.

Because $(\hat \mu_i)_i$ is a geodesic of $X_\infty$, for all $1\leq i \leq n-1$, $(\hat \mu_i)_\#$ is disjoint from $(\hat \mu_{i-1})_\#$ and $(\hat \mu_{i+1})_\#$ (except on $\{\infty\}$). Moreover, $(\hat \mu_{i-1})_\#$ and $(\hat \mu_{i+1})_\#$ intersect outside $\{\infty\}$. Thus $(\hat \mu_i)_\#$ splits the sphere $\Sph^2$ into two connected components, and one of them contains $(\hat \mu_{i-1})_\#$ and $(\hat \mu_{i+1})_\#$. We denote by $A_i$ the other connected component. Observe that for every $1\leq i \leq n-2$, $A_i$ is disjoint from $A_{i+1}$. For every $1\leq i \leq n-1$, we choose a ray $\mu_i$ such that $(\mu_i)_\#$ is included in $A_i$ (such a $\mu_i$ exists because the $\hat \mu_i$'s are essential curves). Hence we have a path $(\mu_i)_{1\leq i\leq n-1}$ of length $(n-1)$ in $X_r$.

Let us now prove that $d(x,\mu_1)\leq 2$: if $(\mu_1)_\#$ intersects $x_\#$, then $x_\#$ is in the connected component of $\Sph^2 - \hat x_\#$ which contains $(\hat \mu_1)_\#$ and $(\mu_1)_\#$. Every representative of a ray which is included in the other connected component of $\Sph^2 - \hat x_\#$ does not intersect neither $(\mu_1)_\#$, nor $x_\#$: thus $d(x,\mu_1)\leq 2$.
We prove in the same way that $d(\mu_{n-1},y)\leq 2$.\end{proof}

\begin{lemma}\label{lemme X_r X_infty}
Let $\hat x \in X_\infty$. Let $x \in X_r$ homotopically disjoint from $\hat x$. Then: $$d(f(x),\hat x) \leq 2.$$
\end{lemma}

\begin{proof}
We denote by $x_\#$ and $\hat x_\#$ the geodesic representatives of $x$ and $\hat x$, which are mutually disjoint (expect on $\{\infty\}$). Because $x_\#$ is disjoint from $\hat x_\#$, there exists an open topological disk $\mathcal{U}$ of $\Sph^2$ which contains $x_\#-\{ \infty\}$ and which is disjoint from $\hat x_\# -\{ \infty\}$. Similarly, because $f(x)_\#$ is disjoint from $x_\#$, there is an open topological disk $\mathcal{V}$ which contains $x_\#-\{ \infty\}$ and is disjoint from $f(x)_\# -\{ \infty\}$. Thus $\mathcal{U} \cap \mathcal{V}$ contains an open topological disk which contains $x_\#-\{ \infty\}$ and which is disjoint from $(\hat x_\# \cup f(x)_\#) -\{ \infty\}$. In particular, $\mathcal{U} \cap \mathcal{V}$ contains points of $K$, because it contains the Cantor-endpoint of $x_\#$. It follows that there exists a simple curve $\hat y_0 \subset (\mathcal{U} \cap \mathcal{V}) - K$ of $\Sph_2$ which contains $\infty$, and whose isotopy class is $\hat y \in X_\infty$. Finally, we have $d(\hat y,\hat x)=d(\hat y ,f(x))=1$, which completes the proof.
\end{proof}

\begin{lemma} \label{quasi-iso}
For all $x,y \in X_r$, we have: $$ d(f(x),f(y))\leq d(x,y)+4.$$
\end{lemma}

\begin{proof}
Let $x,y \in X_r$ and $n=d(x,y)$.
If $x$ and $y$ do not have the same Cantor-endpoint, we choose a geodesic path $(\gamma_i)_{0\leq i \leq n}$ in $X_r$ between $x$ and $y$, and such that for all $i,j$, $\gamma_i$ and $\gamma_j$ do not have the same Cantor-endpoint. Such $\gamma_i$ exist, up to change some of the Cantor-endpoint for a neighbor point of $K$ without adding new intersections with the other $\gamma_k$'s. 

If $x$ and $y$ have the same Cantor-endpoint, we choose for $\gamma_n$ a ray which is homotopically disjoint from $y$ and from $f(y)$, and whose Cantor-endpoint is distinct from the Cantor-endpoint of $x$. We then choose a geodesic path $(\gamma_i)_{0\leq i \leq n}$ in $X_r$ between $x=\gamma_0$ and $\gamma_n$.

We now choose a small neighborhood $N_i$ of each Cantor-endpoint of $\gamma_i$, homeomorphic to a closed disk, which does not intersect any $\gamma_k$ for $k\neq i$, and such that the $N_i$'s are mutually disjoint. Moreover, we assume that the boundary of $N_i$ is disjoint from $K$ for all $i$.
If $y \neq \gamma_n$, we also choose some $\hat \gamma_n$ disjoint from $y$. We define for each $\gamma_i$ a curve $\hat \gamma_i$ as follow: we follow $\gamma_i$ until it crosses $N_i$, we follow the boundary of $N_i$, we follow $\gamma_i$ again to $\infty$. We get this way an element of $X_\infty$.

By construction, for all $i$ between $2$ and $n-1$, we have $d(\hat \gamma_{i-1},\hat \gamma_i)=d(\gamma_i,\gamma_{i+1})=1$.
According to Lemma \ref{lemme X_r X_infty} applied to $\hat \gamma_0$ disjoint from $x=\gamma_0$ and to $\hat \gamma_n$ disjoint from $y$, we get $d(\hat \gamma_0,f(x))\leq 2$ and $d(\hat \gamma_n,f(y))\leq 2$. Finally, we have $d(f(x),f(y))\leq n+4$.
\end{proof}

\paragraph{End of the proof of Proposition \ref{prop quasi-isom}:}
Lemmas \ref{retour} (for the first inequality) and \ref{quasi-iso} (for the second inequality) give us: $$d(x,y) -2 \leq d(f(x),f(y)) \leq d(x,y)+4.$$
Lemma \ref{lemme X_r X_infty} gives us a constant $C=2$ such that every $\hat x$ in $X_\infty$ is in a $C$-neighborhood of $f(X_r)$, which ends the proof of Proposition \ref{prop quasi-isom}. \qed

\subsubsection*{Gromov-hyperbolicity of the ray graph}

Finally, we have proved the following theorem:
\begin{theo}\label{theo-hyperbolique}
The ray graph is Gromov-hyperbolic.
\end{theo}

\begin{proof}
This is a consequence of Proposition \ref{prop quasi-isom} (there exists a quasi-isometric embedding of $X_r$ in $X_\infty$), of Proposition \ref{G_infini hyp} ($X_\infty$ is Gromov-hyperbolic) and of Theorem \ref{bridson} (if $f:X\rightarrow X'$ is a quasi-isometric embedding and if $X'$ is Gromov-hyperbolic, then $X$ is Gromov-hyperbolic).
\end{proof}

\section{Non trivial quasimorphisms}\label{section-qm}

In \cite{Bestvina-Fujiwara}, Mladen Bestvina and Koji Fujiwara have proved that the space of classes of non trivial quasimorphisms on the mapping class group of a compact surface has infinite dimension. They first proved (Theorem $1$ of \cite{Bestvina-Fujiwara}) that if a group $G$ acts by isometries on a Gromov-hyperbolic space $X$, then, assuming the existence of some loxodromic elements satisfying some specific properties in $G$, the space of classes of non trivial quasimorphisms on $G$ has infinite dimension.
As a second step, they proved that if the action of $G$ on $X$ is \emph{weakly properly discontinuous} (WPD), then there exist some loxodromic elements satisfying the hypothesis of Theorem $1$. Finally, they proved that the action of the mapping class group of a compact surface on the curve complex is WPD. \\

An element $g$ of a group $G$ is said to act weakly properly discontinuously on a space $X$ if for all \mbox{$x\in X$,} for all $C >0$, there exists $N>0$ such that the number of $\sigma \in G$ satisfying $d(x,\sigma x) \leq C$ and $d(g^N x,\sigma g^N x)\leq C $ is finite (see for example Calegari \cite{SCL} p74).

\begin{claim} \label{rq-WPD} For all $g\in \Gamma= \MCG(\R^2 -\Cantor)$, the action of $g$ on the ray graph $X_r$ is not weakly properly discontinuous. \end{claim}

Indeed, for all $x \in X_r$, for all $N\in \N$, there are infinitely many \mbox{$\sigma \in \Gamma$} such that $d(x,\sigma x) \leq 1$ and $d(g^N x,\sigma g^N x)\leq 1$: let $\mathcal{U}$ be a neighborhood of a point of the Cantor set such that $\mathcal{U}$  is disjoint from $x$ and from $g^N x$. Then every $\sigma$ supported in $\mathcal{U}$  fixes $x$ and $g^N x$, and thus satisfies $d(x,\sigma x) \leq 1$ and $d(g^N x,\sigma g^N x)\leq 1$. Moreover, there are infinitely many such $\sigma$, because there are infinitely many points of the Cantor set in $\mathcal{U}$. \qed

It follows that the strategy of \cite{Bestvina-Fujiwara} cannot be directly apply to our setting. However, we can find explicit elements of $\Gamma$ which satisfy the hypothesis of Theorem $1$ of \cite{Bestvina-Fujiwara}, which allow us to prove that the space of classes of non trivial quasimorphisms on $\Gamma$ has infinite dimension.\\

We first find an element $h\in \Gamma$ which acts by translation on the axis $(\alpha_k)_{k}$ previously defined (in Section \ref{section1}). We then prove, using a "number of positive intersections", that if $w$ is a sufficiently long segment of this axis, then for all $g\in \Gamma$, $g$ can't reverse this segment in a close neighborhood of the axis (Proposition \ref{copies}). Finally, we use this Proposition to, on the one hand, construct an explicit non trivial quasimorphism on $\Gamma$, and on the second hand, construct elements of $\Gamma$ that satisfy the hypothesis of Theorem $1$ of \cite{Bestvina-Fujiwara}.

\subsection{A loxodromic element of $\Gamma$}

We want to define a loxodromic element $h \in \Gamma$ as in Figure \ref{figu:h}, where each string sends a subset of the Cantor set to another one, in such a way so that $h$ maps $\alpha_k$ on $\alpha_{k+1}$ for all $k \in \N$ (see Figure \ref{figu:action-h}).

\begin{figure}[!h]
\labellist
\small\hair 2pt
\pinlabel $K_{-4}$ at 4 204
\pinlabel $K_{-3}$ at 56 204
\pinlabel $K_{-2}$ at 106 204
\pinlabel $K_{-1}$ at 156 204
\pinlabel $K_0$ at 264 204
\pinlabel $K_1$ at 312 204
\pinlabel $K_2$ at 362 204
\pinlabel $K_3$ at 413 204
\pinlabel $\infty$ at 211 204
\pinlabel $h$ at 465 110
\endlabellist
\centering
\vspace{0.4cm}
\includegraphics[scale=0.5]{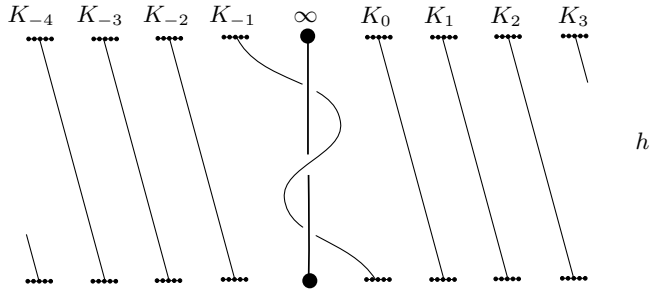}
\caption{Representation of $h \in \Gamma$.}
\label{figu:h}
\end{figure}

Because $(\alpha_k)_{k\in \N}$ is a geodesic half-axis (according to Proposition \ref{alpha_k geod}), we will have that $(h^n(\alpha_0))_{n\in \Z}$ is a geodesic axis of the ray graph, on which $h$ acts by translation.

\begin{figure}[!h]
\labellist
\small\hair 2pt
\pinlabel $K_{-1}$ at 155 317
\pinlabel $K_0$ at 264 317
\pinlabel $K_1$ at 314 317
\pinlabel $\infty$ at 214 317
\pinlabel $h$ at 451 256
\pinlabel $h$ at 451 121
\pinlabel $\alpha_0$ at 239 281
\pinlabel $\alpha_1$ at 309 154
\pinlabel $\alpha_2$ at 351 12
\endlabellist
\centering
\vspace{0.7cm}
\includegraphics[scale=0.5]{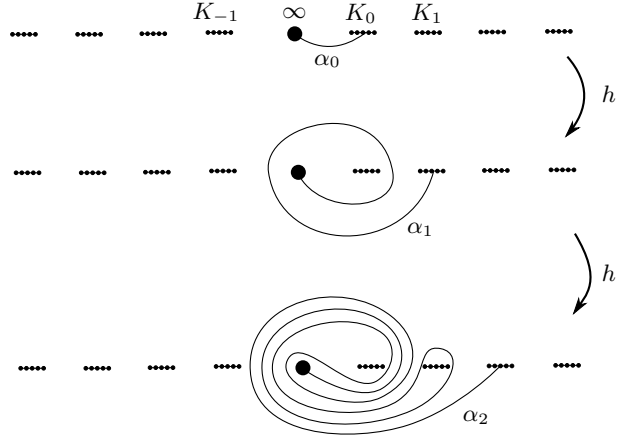}
\caption{Action of $h$ on the rays $\alpha_0$ and $\alpha_1$.}
\label{figu:action-h}
\end{figure}

\subsubsection*{Definition of $h$}

We fix an equator $\E$ and some alphabet of segments $(s_k)_{k\in \Z}$ as in Subsection \ref{partie codage}. For all $k \in \Z- \{0\}$, we denote by $K_k$ the points of $K$ between $s_{k-1}$ and $s_k$ on $\E$. In particular, the $K_k$'s are clopen subsets of the initial Cantor set $K$ for all $k$, thus they also are Cantor sets (any clopen subset of a Cantor set also is a Cantor set, according to the characterization of Cantor sets as totally disconnected compact sets without isolated point). We denote by $I$ a connected component of $\E - K$ such that $I \cup \{\infty\}$ splits the equator into two connected components, so that one of them contains all the segments $s_k$'s with $k>0$, and the other contains all the segments $s_k$'s with $k<0$.

\begin{figure}[!h]
\labellist
\small\hair 2pt
\pinlabel $\infty$ at 82 10
\pinlabel $\infty$ at 355 10
\pinlabel $I$ at 79 200
\pinlabel $I$ at 356 200
\pinlabel $\Cc_N$ at 25 190
\pinlabel $\Cc_S$ at 300 190
\pinlabel $K_0$ at 133 20
\pinlabel $K_0$ at 406 21
\pinlabel $K_{-1}$ at 25 27
\pinlabel $K_{-1}$ at 299 28
\pinlabel $K_1$ at 175 67
\pinlabel $K_1$ at 447 68
\endlabellist
\centering
\vspace{0.3cm}
\includegraphics[scale=0.6]{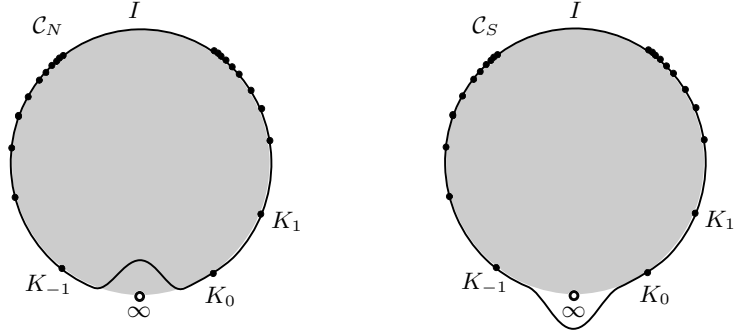}
\caption{$\Cc_N$ and $\Cc_S$ (the grey area is the northern hemisphere).}
\label{cercles}
\end{figure}

Let $\Cc_N$ be a topological circle which coincides with the equator $\E$ outside a neighborhood of infinity, and which goes through the northern hemisphere above infinity. Let $\Cc_S$ be a topological circle which coincides with the equator $\E$ outside a neighborhood of infinity, and which goes through the southern hemisphere below infinity (see Figure \ref{cercles}).

Let $\tilde t_1$ be a homeomorphism of $\Cc_N$ which maps the Cantor subset $K_k$ on $K_{k+1}$ for all $k\in \Z$, and which is the identity on $I$. We extend $\tilde t_1$ to a homeomorphism of the sphere $\Sph^2$ fixing infinity, and we consider its isotopy class $t_1 \in \Gamma$.

\vspace{0.3cm}
\begin{figure}[!h]
\labellist
\small\hair 2pt
\pinlabel $t_1$ at 456 145
\pinlabel $t_2$ at 456 88
\pinlabel $t_1$ at 456 31
\endlabellist
\centering
\includegraphics[scale=0.5]{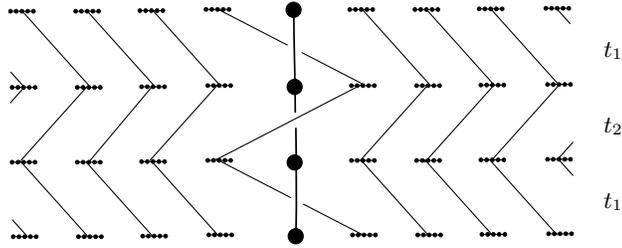}
\caption{Definition of $h:=t_1t_2t_1$.}
\label{figu:3h}
\end{figure}

Similarly, let $\tilde t_2$ be a homeomorphism of $\Cc_S$ which maps the Cantor subset $K_{k+1}$ on $K_{k}$ for all $k\in \Z$, and which is the identity on $I$. We extend $\tilde t_2$ to a homeomorphism of the sphere $\Sph^2$ fixing infinity, and we consider its isotopy class $t_2 \in \Gamma$.
In particular, if we denote by $\phi$ the isotopy class of the rotation of angle $\pi$ around $\infty$ and which maps for all $k\in \Z$ the Cantor subset $K_k$ on $K_{-k-1}$, then we can choose $t_2 = \phi t_1 \phi^{-1}$.

Finally, we set $h:=t_1 t_2 t_1$ (see Figure \ref{figu:3h}).

\subsubsection*{Action of $h$ on the ray graph}

Recall that if there exists a geodesic axis of $X_r$ which is preserved by some isometry $g$, and if $g$ has no fixed point on this axis, then the action of $g$ on $X_r$ is said to be \emph{loxodromic}, and this geodesic axis is called the axis of $g$.

\begin{theo}\label{theo-halphak}
The action of $h$ on the ray graph is loxodromic, with axis $(\alpha_k)_k$. More precisely, $h(\alpha_k)=\alpha_{k+1}$ for all $k\in \N$.
\end{theo}

To see that $h(\alpha_k)=\alpha_{k+1}$ for all $k\geq 0$, we represent $\alpha_k$ with a graph as in Figure \ref{diag-alpha2}. 

\begin{proof}
\begin{figure}[!h]
\labellist
\small\hair 2pt
\pinlabel $3$ at 365 -5
\pinlabel $1$ at 405 0
\pinlabel $1$ at 415 60
\endlabellist
\centering
\includegraphics[scale=0.7]{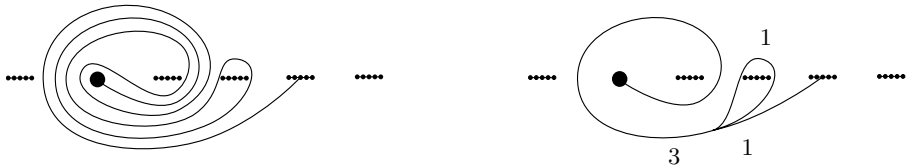}
\vspace{0.2cm}
\caption{On the left, the ray $\alpha_2$; on the right, a graph representing it.}
\label{diag-alpha2}
\end{figure}

For each ray, we can choose a curve representing it and identify some subsegments of the curve which stay close to each other. We get this way a finite graph smoothly embedded in $\Sph^2$ and disjoint from all points of $K$ except the Cantor-endpoint of the initial ray. At each vertex, the edges go in two directions. Each edge has a weight, which correspond to the number of subsegments that it carries: at each vertex, in one of the two directions there is only one edge, and its weight is the sum of the weights of the edges going in the opposite direction. We can deduce the initial ray from a graph representing it: indeed, it is enough to duplicate each edge the number of times corresponding to its weight, and then to glue the subsegments in the unique possible way. We only glue together subsegments which goes in opposite directions, and we want to draw a simple curve, thus there is a well defined way to glue the subsegments together.

\begin{figure}[!h]
\labellist
\small\hair 2pt
\pinlabel $K_k$ at 270 78
\pinlabel $1$ at 235 85
\pinlabel $3^{k-3}$ at 170 86
\pinlabel $3^{k-2}$ at 127 86
\pinlabel $1$ at 240 5
\pinlabel $3^{k-3}$ at 160 5
\pinlabel $3^{k-2}$ at 110 5
\pinlabel $3^{k-1}$ at 60 5
\endlabellist
\centering
\includegraphics[scale=0.55]{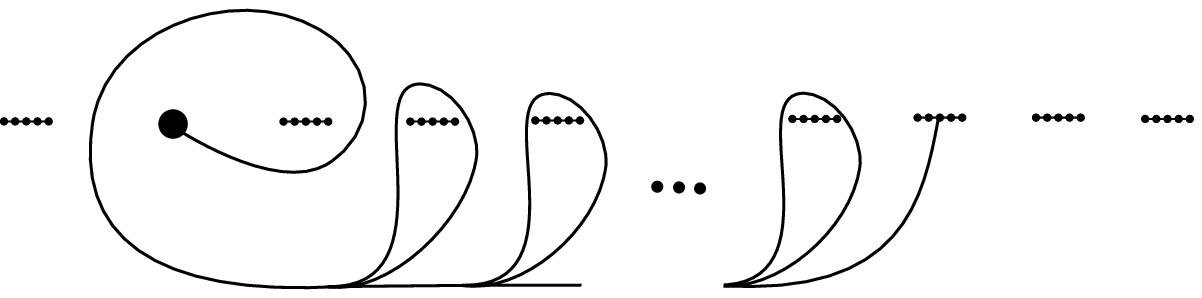}
\caption{Example of a graph representing the ray $\alpha_k$.}
\label{diag-alphak}
\end{figure}

On Figure \ref{diag-alphak}, we draw some specific graph representing $\alpha_k$, for every $k \geq 0$. As there exists a curve $a_k$ representing $\alpha_k$ which stays in a tubular neighborhood of this graph, if $h_0$ is a representative of $h$, we have that $h_0(a_k)$ stays in a tubular neighborhood of the image of the graph by $h_0$: the ray corresponding to the image of the graph is $h(\alpha_k)$.

\begin{figure}[!h]
\labellist
\small\hair 2pt
\pinlabel $t_1$ at 455 404
\pinlabel $t_2$ at 454 267
\pinlabel $t_1$ at 454 134

\pinlabel $\alpha_k$ at 424 456
\pinlabel $t_1(\alpha_k)$ at 429 326
\pinlabel $t_2t_1(\alpha_k)$ at 429 196
\pinlabel $h(\alpha_k)$ at 429 56

\pinlabel $1$ at 295 405
\pinlabel $3^{k-3}$ at 210 405
\pinlabel $3^{k-2}$ at 165 405
\pinlabel $3^{k-1}$ at 115 405
\pinlabel $1$ at 285 480
\pinlabel $3^{k-3}$ at 220 483
\pinlabel $3^{k-2}$ at 182 483

\pinlabel $1$ at 320 13
\pinlabel $3^{k-3}$ at 240 13
\pinlabel $3^{k-2}$ at 200 13
\pinlabel $3^{k-1}$ at 140 13
\pinlabel $1$ at 317 90
\pinlabel $3^{k-3}$ at 257 92
\pinlabel $3^{k-2}$ at 218 92

\pinlabel $K_k$ at 320 475
\pinlabel $K_{k+1}$ at 354 83
\endlabellist
\centering
\vspace{0.3cm}
\includegraphics[scale=0.52]{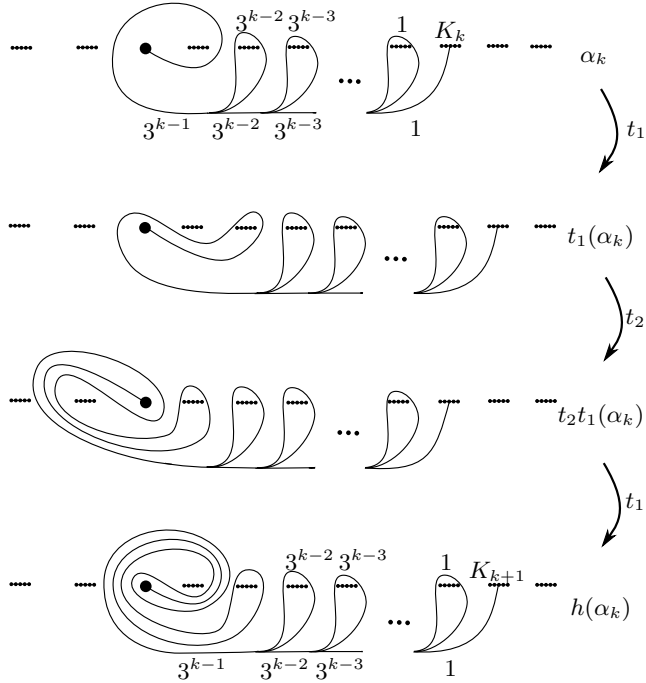}
\caption{Image of $\alpha_k$ by $h$.}
\label{diag-h-alphak}
\end{figure}

On Figure \ref{diag-h-alphak}, we draw a graph representing $\alpha_k$ and the successive images of this graph by the representatives of $t_1$, $t_2$ and $t_1$. Thus the final graph is the image of the graph of $\alpha_k$ by $h$: it represents $h(\alpha_k)$. Moreover, we see that the ray represented by this graph is $\alpha_{k+1}$: hence $h(\alpha_k)=\alpha_{k+1}$ for all $k \in \N$.
\end{proof}

\subsection{Number of positive intersections}

We still denote by $X_r$ the ray graph, and we orient each ray from infinity to its Cantor-endpoint.

\begin{figure}[!h]
\labellist
\small\hair 2pt
\pinlabel $\alpha$ at 150 50
\pinlabel $\beta$ at 200 3
\pinlabel $\alpha$ at 3 50
\pinlabel $\beta$ at 53 3
\endlabellist
\centering
\includegraphics[scale=0.7]{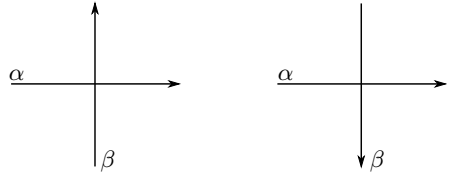}
\caption{Positive intersection on the left, negative intersection on the right}
\label{intersections}
\end{figure}

\begin{definition}[Number of positive intersections] Let $I : X_r^2 \rightarrow \N \cup \{\infty\}$ be the map which sends any couple of oriented rays $(\alpha,\beta)$ to the number of positive intersections between two representatives of $\alpha$ and $\beta$ which are in minimal position (see Figure~\ref{intersections}).
\end{definition}

\begin{remark}
\begin{enumerate}
\item This number is well defined: it does not depend on the choice of representatives for $\alpha$ and $\beta$ (according to Proposition \ref{prop 3.5}).
\item We do not necessary have $I(\alpha,\beta) = I(\beta,\alpha)$.
\item For all $g \in \Gamma$, $I(g \cdot \alpha,g \cdot \beta)=I(\alpha,\beta)$ (because $\Gamma$ is the quotient of the group of homeomorphisms which preserve the orientation). \label{rq3}
\end{enumerate}
\end{remark}

\subsubsection*{Case of the sequence $(\alpha_k)_k$}

\begin{lemma} \label{lemme-I1}
Let $\beta$ and $\gamma$ be two elements of $X_r$ such that $A(\gamma)\leq A(\beta)-2$, where $A$ is the map defined at Section \ref{def de A}. Then $I(\gamma,\beta)\geq 1$.
\end{lemma}

\begin{figure}[!h]
\labellist
\small\hair 2pt
\pinlabel $\mathring \alpha_n$ at 5 50
\pinlabel $s_{-1}$ at 131 65
\pinlabel $s_{-1}$ at -10 245
\pinlabel $s_0$ at 127 245
\pinlabel $s_1$ at 127 210
\pinlabel $\mathring \gamma$ at 100 229
\pinlabel $\mathring \gamma$ at 20 229
\pinlabel ${p}_{n-1}$ at 70 10
\endlabellist
\centering
\vspace{-0.2cm}
\includegraphics[scale=0.5]{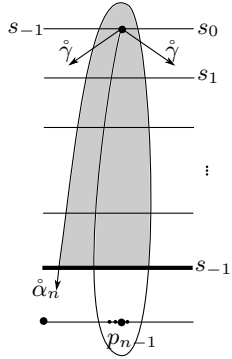}
\caption{Illustration of Lemma \ref{lemme-I1} (by definition of $(\alpha_k)_k$, there is no point of $K$ in the grey area).}
\label{figu:lemmeI1}
\end{figure}

\begin{proof}
Let $n:=A(\beta)$. Then $\gamma$ does not begin with $\mathring \alpha_{n-1}$. On Figure \ref{figu:lemmeI1}, we represent the beginning of $\beta$, i.e. $\mathring \alpha_n$. Every representative of $\gamma$ starts at infinity and has to go to its endpoint on the Cantor set: thus every representative of $\gamma$ has to leave the grey area. Because $\gamma$ does not begin with $\mathring \alpha_{n-1}$, $\gamma$ cannot leave the grey area by crossing $s_{-1}$. Thus $\gamma$ leaves the grey are by crossing $\beta$. The first intersection is positive, hence $I(\gamma,\beta)\geq 1$.
\end{proof}

\begin{remark} \begin{enumerate}
\item Because $\alpha_0$ and $\alpha_1$ are homotopically disjoint, we have: \\
$I(\alpha_0,\alpha_1)=I(\alpha_1,\alpha_0)=0$ ;
\item We won't use this result, but note that it is possible to compute precisely the number of intersections between $\alpha_0$ and $\alpha_k$ for all $k \geq 2$. We have:
$$I(\alpha_0,\alpha_k)={{3^{k-1}+2k-3}\over 4}~ ~ ~ ~\mathrm{and} ~ ~ ~ ~ I(\alpha_k,\alpha_0) = {{3^{k-1}-2k+1}\over 4}.$$
 Indeed, denote by $(p_k,n_k)=(I(\alpha_0,\alpha_k),I(\alpha_k,\alpha_0))$. We have: $$(p_{k+1},n_{k+1})=(2p_k+n_k+1,p_k+2n_k).$$ This comes from the construction of $(\alpha_k)_k$: we draw a tube around $\alpha_{k-1}$, and thus we can look at the orientation of the intersections of this tube and $\alpha_0$. We then know how to write $p_k$ and $n_k$ in terms of $k$.
\end{enumerate}
\end{remark}

\subsection{Non-reversibility of the axis $(\alpha_k)_k$}

We denote by $B$ the ceiling of the $(2,4)$-Morse constant of the ray graph (see Section \ref{def-q.i.}). We want to show some non-reversibility property for the axis $(\alpha_k)_k$ (Proposition \ref{copies}), which will be fundamental for our constructions of non trivial quasimorphisms (Proposition \ref{prop-qm} and Theorem \ref{dim infinie}). To prove that property, we first need to be able to compare the orientations of some segments.

\subsubsection*{Segments with the same orientation}

Let $X$ be a geodesic metric space. Let $\gamma_1=[p_1q_1]$ and $\gamma_2=[p_2q_2]$ be two geodesic segments of $X$ with the same length, and oriented from $p_i$ to $q_i$. Let $\gamma'_1$ be a geodesic segment (possibly of infinite length) which contains $\gamma_1$. Let $C>0$ such that $\gamma_2$ is included in the $C$-neighborhood of $\gamma'_1$, and such that $d(p_1,p_2)\leq C$. Moreover, we assume that $|\gamma_1|=|\gamma_2|\geq 3C$. 
In this setting, we say that $\gamma_1$ and $\gamma_2$ have \emph{the same orientation} if for all $r\in \gamma'_1$ such that $d(r,q_2)\leq C$, $r$ is on the same side of $p_1$ than $q_1$ on $\gamma'_1$. Because the length of $\gamma_1$ and $\gamma_2$ is greater than or equal to $3C$, we can check that the existence of only one $r$ satisfying this conditions is enough.

\begin{lemma}\label{géod-dist}
If $\gamma_1$ and $\gamma_2$ are segments satisfying the previous properties and which have the same orientation, then $d(q_1,q_2)\leq 3C$.
\end{lemma}

\begin{proof}
Let $r$ on $\gamma'_1$ such that  $d(q_2,r)\leq C$. We denote by $\alpha$ the segment of $\gamma'_1$ between $p_1$ and $r$, and by $\beta$ the segment between $r$ and $q_1$. 

\textbullet\ \textbf{First case: if $r \in \gamma_1$.} We have: $$|\gamma_2|=|\gamma_1|=|\alpha|+|\beta|\leq d(p_1,p_2)+|\alpha|+C.$$
Hence: $$|\beta| \leq d(p_1,p_2)+C \leq 2C.$$
Finally, we get: $$d(q_1,q_2) \leq d(q_1,r)+d(r,q_2) \leq |\beta| +C \leq 3C.$$

\textbullet\ \textbf{Second case: if $r \notin \gamma_1$.} The segment $[p_1,r]\subset \gamma'_1$ contains $\gamma_1$ (because $\gamma_1$ and $\gamma_2$ have the same orientation, thus $r$ cannot be on the other side of $p_1$ on $\gamma'_1$). Hence: $$|\gamma_1|+|\beta| \leq d(p_1,p_2) + |\gamma_2| + d(q_2,r) \leq |\gamma_1|+2C.$$
Thus: $$|\beta| \leq 2C.$$
Finally, we get: $$d(q_1,q_2) \leq d(q_1,r)+d(r,q_2) \leq |\beta| + C \leq 3C.$$
\end{proof}

\subsubsection*{Non reversibility}

\begin{prop} [Non reversibility]\label{copies}
Let $B$ be the ceiling of the $(2,4)$-Morse constant of the ray graph, and let $w$ be a subsegment of the axis $l=(\alpha_k)_{k\in \Z}$ of length greater than $10B$. For all $g \in \MCG(\R^2-K)$, if $g \cdot w$ is included in the $B$-neighborhood of $l$, then it has the same orientation than $w$.
\end{prop}

\noindent  In other words, the subsegments of the axis $l$ whose length is greater than $10B$ are \emph{non reversible}: there is no copy of $w^{-1}$ with the same orientation than $w$ in the $B$-neighborhood of $l$.

\begin{remark} If an element $h'\in \Gamma$ is conjugated to $h^{-1}$ by a map {$\psi$}, let us denote by $l'$ the image of $l$ by $\psi$, equipped with the opposite orientation of $l$. This is an axis for $h'$.
According to the previous proposition, for all subsegment $w$ of the axis $l'$ of $h'$ whose length is greater than $10B$, and whose orientation is the same than the orientation of $l'$, for all $g \in \Gamma$, if $g \cdot w$ is included in the $B$-neighborhood of the axis $l$, then $g \cdot w$ has the opposite orientation than $l$.
\end{remark}

\paragraph{Proof of Proposition \ref{copies}.}
We prove the two following lemmas, which allow us to conclude:

\begin{lemma} \label{w1}
Let $m<n$ be two positive integers and let $w=(\alpha_m,\alpha_{m+1},...,\alpha_n)$ be a subsegment of $(\alpha_k)_{k\in \N}$. Let $g \in \MCG(\R^2-K)$ such that $d(\alpha_m,g \cdot \alpha_n)\leq B$, and such that $g \cdot w$ is in the $B$-neighborhood of $l$, with the opposite orientation than $w$. 
If $|w|>8B+1$, then there exists $m \leq i \leq n$ such that $A(g \cdot \alpha_{i+2})=A(g \cdot \alpha_i)-2$.
\end{lemma}

\begin{proof}
Because $d(\alpha_m,g \cdot \alpha_n)\leq B$, we have $A(g \cdot \alpha_n) \leq m+B$ (according to Corollary~\ref{dist}).\\
Because $g \cdot w^{-1}$ and $w$ have the same orientation and the same length, according to Lemma \ref{géod-dist}, we have:
$$d(\alpha_n,g \cdot \alpha_m)\leq 3B.$$
Hence $A(g \cdot \alpha_m) \geq n-3B$ (according to Corollary \ref{dist}).\\

Because $A$ is $1$-lipschitz (Lemma \ref{lemme distance}), $A(g \cdot w)$ is surjective on the set of integers between $m+B$ and $n-3B$. By contradiction, if we assume that for all $i$ between $m$ and $n$,
$A(g \cdot \alpha_{i+2}) \neq A(g \cdot \alpha_i)-2$, then for all $i$ we have: $$A(g \cdot \alpha_{i+2}) \geq A(g \cdot \alpha_i)-1.$$
By induction, we deduce:
$$A(g \cdot \alpha_n)\geq A(g \cdot \alpha_m)-{n-m\over 2}.$$
Since $A(g \cdot \alpha_m) \geq n-3B$ and $A(g \cdot \alpha_m) \geq n-3B$, we have:
$$m+B\geq n-3B-{n-m\over 2}.$$
Hence:
$$8B \geq n-m.$$
Because we have assumed that $|w| > 8B+1$, we are lead to a contradiction.\end{proof}

\begin{lemma} \label{w2}
For all $g \in \MCG(\R^2-Cantor)$ and for all $i \geq 0$, we have:
$$A(g \cdot \alpha_{i+2}) \neq A(g \cdot \alpha_i)-2.$$
\end{lemma}

\begin{proof}
For all $f \in \MCG(\R^2 -Cantor)$ and for all $\beta, \gamma \in X_r$, we have $I(f.\beta,f.\gamma)=I(\beta,\gamma)$. Hence: $$I(g \cdot \alpha_{i+2},g \cdot \alpha_{i})=I(\alpha_2,\alpha_0)=0.$$

\noindent By contradiction, if $A(g \cdot \alpha_{i+2}) = A(g \cdot \alpha_i)-2$, according to Lemma \ref{lemme-I1} we have: 
$$I(g \cdot \alpha_{i+2},g \cdot \alpha_{i}) \geq 1. $$
This ends the proof of Lemma \ref{w2}.
\end{proof}

The proof of Proposition \ref{copies} follows:

\begin{proof}
By contradiction: assume there exists a copy $w^{-1}$ as wanted, i.e. some $g \in \Gamma$ such that the segment $g \cdot w^{-1}=(g \cdot \alpha_n,...,g \cdot \alpha_{m+1},g \cdot \alpha_m)$ is in the $B$-neighborhood of the axis $l$, and has the same orientation than $w$. Up to compose $g$ by $h^k$ for some $k \in \Z$ if necessary, we can assume that $d(\alpha_m,g \cdot \alpha_n) \leq B$. Because $|w| > 8B+1$, Lemmas \ref{w1} and \ref{w2} give us the conclusion.
\end{proof}

\subsection{An explicit non trivial quasimorphism on~$\Gamma$}

Let us first recall the construction of Fujiwara \cite{Fujiwara} of quasimorphisms acting on Gromov-hyperbolic spaces. We fix some $p\in X_r$. Let $w$ and $\gamma$ be two paths in $X_r$. A \emph{copy of $w$} is a path of the form $g \cdot w$, with $g\in \Gamma$. We denote by $|\gamma|_w$ the maximal number of disjoint copies of $w$ on $\gamma$, and define: $$c_w(g):=d(p,g \cdot p)-\mathrm{inf}_\gamma (\mathrm{long}(\gamma)-|\gamma|_w),$$ where the infimum is taken over all the paths $\gamma$ between $p$ and $g \cdot p$. Because $X_r$ is hyperbolic, the map $q_w : \Gamma \rightarrow \R$ defined by $q_w:=c_w -c_{w^{-1}}$ is a quasimorphism on $\Gamma$ (Proposition $3.10$ of \cite{Fujiwara}). Moreover, the homogeneous quasimorphism $\tilde q_w$ defined by $\tilde q_w(g) = \lim_{n\rightarrow \infty} {q(g^n)\over n} $ does not depend on the choice  of $p \in X_r$.\\

Let us now prove the following Proposition, which will not be useful to prove that the space of classes of non trivial quasimorphisms has infinite dimension.

\begin{prop}\label{prop-qm}
Let $(\alpha_k)_{k\in \Z}$ be the geodesic axis in the ray graph previously defined. Let $w$ be a subsegment of this geodesic of length greater than $10B$, where $B$ is the ceiling of the $(2,4)$-Morse constant of the ray graph.
The quasimorphism $\tilde q_w$ given by Fujiwara's construction is non trivial.
\end{prop}

\begin{remark} Since we know the hyperbolicity constant of the graph $X_\infty$, we can deduce from it the hyperbolicity constant on the ray graph, and thus compute $B$: hence the segment $w$ can be explicitly chosen. \end{remark}

\begin{proof}
Since $\tilde q_w$ is homogeneous, it suffices to show that it is not a morphism in order to prove that it is non trivial. 
We first prove that $\tilde q_w(h)$ is not equal to zero, where $h=t_1t_2t_1$ is the loxodromic element of $\Gamma$ previously defined. We then show that $\tilde q_w(t_1)=\tilde q_w(t_2)=0$: thus $\tilde q_w(t_1t_2t_1) \neq \tilde q_w(t_1)+\tilde q_w(t_2)+\tilde q_w(t_1)$, i.e. $\tilde q_w$ is not a morphism.

The first affirmation is a consequence of Proposition \ref{copies}. This is the strategy described by Calegari in \cite{SCL}, page $74$: if we denote by $m$ the length of $w$ and if we choose $p =\alpha_0$, for all $k\in \N$ we have $c_w(h^{km})=k$ and $c_{w^{-1}}(h^{km})=0$. Indeed, the first inequality is clear, and for the second one, we use Lemma $3.3$ of \cite{Fujiwara}: the paths which realize the infimum are $(2,4)$-geodesics. Thus they stay in a $B$-neighborhood of the axis $(\alpha_k)_k$, according to the Morse Lemma (\ref{Morse}). Moreover, this neighborhood does not contain any copy of $w^{-1}$, according to Proposition \ref{copies} (see \cite{SCL}, Section $3.5$ for more details).
We have: $$\tilde q_w(h^m) := \lim_{k\rightarrow \infty} {c_w(h^{km}) - c_{w^{-1}}(h^{km})\over k} = 1.$$

Let us now prove that $\tilde q_w(t_1)=\tilde q_w(t_2)=0$. We choose $p=\alpha_0$. For all \mbox{$k \in \N$,} $t_1^k(\alpha_0)$ is the isotopy class of a curve included in the northern hemisphere. Hence \mbox{$d(p,t_1^k \cdot p)=1$.} We have $c_w(t_1^k)=c_{w^{-1}}(t_1^k)=0$, thus $\tilde q_w(t_1)=0$. Similarly, $\tilde q_w(t_2)=0$. Finally, we have proved that $\tilde q_w$ is a non trivial quasimorphism. \end{proof}

\begin{remark} To prove that $\tilde q_w$ is not a morphism, we can also prove that $\Gamma$ is a \emph{perfect group}, i.e. that every element of $\Gamma$ is a product of commutators. The only morphism from a perfect group to $\R$ is the trivial morphism. Since $\tilde q_w$ is not identically zero, it is not a morphism. 
\end{remark}

We deduce from a Lemma of Calegari in \cite{blog-Calegari} that $\Gamma$ is a perfect group. The Lemma claims that if $g \in \Gamma$ is such that there exists $x \in X_r$ so that $d(x,gx)=1$, then $g$ is a product of at most $2$ commutators.

Let $g\in \Gamma$ and $x\in X_r$. We consider a path in $X_r$ between $x$ and $gx$, and we denote it by $(x=x_0,x_1,...,x_n=gx)$. Since $\Gamma$ acts transitively on $X_r$, for all $1\leq i \leq n-1$ there exists $g_i \in \Gamma$ which maps $x_{i+1}$ on $x_i$: thus $g_i$ is the product of at most two commutators. It follows that $g_1...g_{n-1}g$ maps $x$ on $x_1$, with $d(x,x_1)=1$. Hence this element is also a product of at most two commutators. Finally, $g$ is a product of (at most $2n$) commutators.

\subsection{Dimension of the space of classes of non trivial quasimorphisms on~$\Gamma$}

\begin{theo}\label{dim infinie}
The space $\tilde{Q}(\Gamma)$ of classes of non trivial quasimorphisms on $\Gamma$ has infinite dimension.
\end{theo}

\begin{proof}
We use Theorem $1$ of Bestvina \& Fujiwara \cite{Bestvina-Fujiwara}. Since $\Gamma$ acts by isometries on the ray graph which is hyperbolic, it suffices to find two loxodromic elements $h_1, h_2 \in \Gamma$ acting by translation on their axes $l_1$ and $l_2$, such that $l_1$ and $l_2$ have the orientation induced by this action, and which satisfy the two following properties (see \cite{Bestvina-Fujiwara}) :

\begin{enumerate}
\item "$h_1$ and $h_2$ are independent": the distance between any two half axes of $l_1$ and $l_2$ is unbounded.
\item "$h_1  \nsim h_2$": there exists a constant $C$ such that for every segment $w$ of $l_2$ whose length is greater than $C$, for every $g \in \Gamma$, either $g \cdot w$  goes out the $B$-neighborhood of $l_1$, or it has the opposite orientation.
\end{enumerate}

\begin{figure}[!h]
\centering
\def\svgwidth{0.86\textwidth}
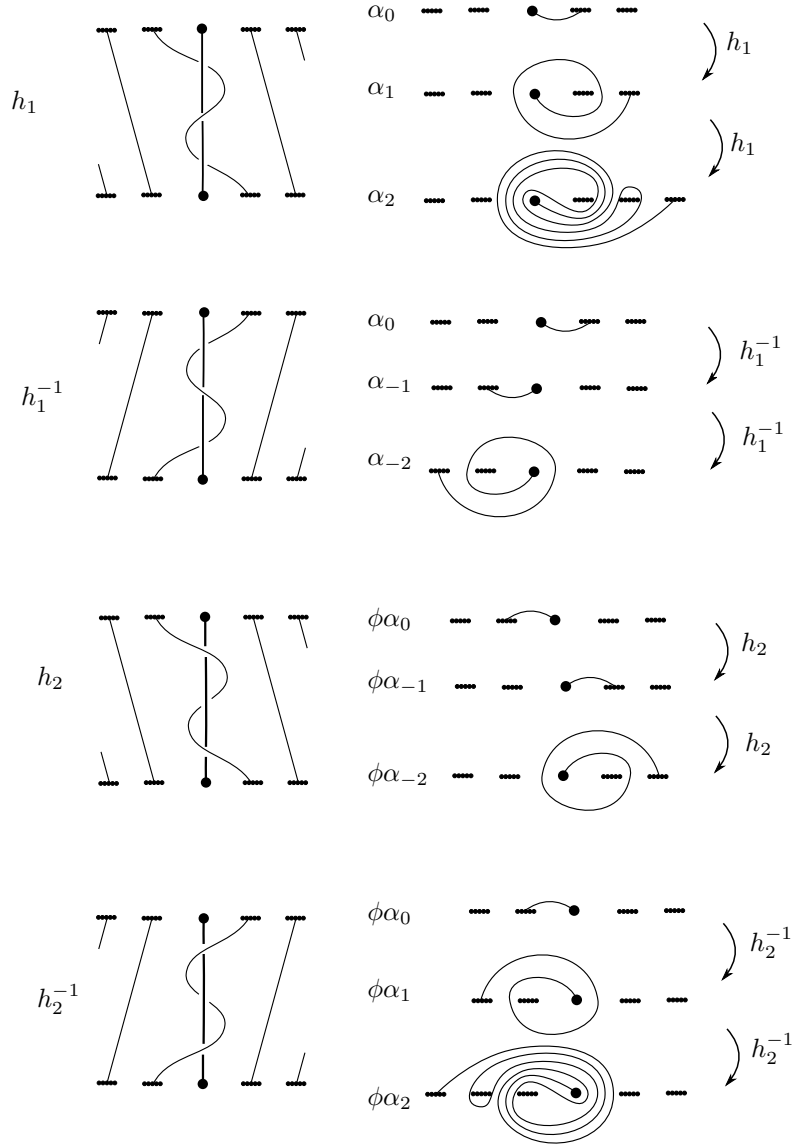
\caption{$h_1$, $h_2$, their inverses and their actions on some rays}
\label{h et compagnie}
\end{figure}

\noindent Let us find two such loxodromic elements. We denote by $h_1 \in \Gamma$ the element $h$ previously defined (which acts by translation on the axis $(\alpha_k)_k$). Let $\phi \in \Gamma$ be the isotopy class of the $\pi$-rotation around infinity. We assume that $K$ is symmetrically embedded around $\infty$, so that $\phi$ preserves $K$ and maps each Cantor subset $K_i$ on $K_{-i-1}$. 
Finally, let  $h_2:= \phi h_1^{-1} \phi^{-1}$. Then $h_1 \nsim h_2$ according to Proposition \ref{copies} and the Remark which follows it (the constant $C:=10B$ works, where $B$ is the ceiling of the Morse constant). Moreover, we will prove that $h_1$ and $h_2$ are independent, which will conclude the proof.

We proved in Corollary \ref{dist} that for every $n \geq 2$, every ray which is in the $(n-2)$-neighborhood of $h_1^n(\alpha_0)$ begins like $\mathring \alpha_2$. Similarly, every ray which is in the $(n-1)$-neighborhood of $h_1^{n-1}(\alpha_0)$ begins like $\mathring \alpha_1$. 

A similar behavior is true for $h_2$, $h_2^{-1}$ and $h_1^{-1}$. We denote by $\sigma$ the isotopy class of the axial symmetry along the equator. In particular, $\sigma$ is equal to its inverse, fixes the Cantor set, and is not an element of $\Gamma$ (because it does not preserve the orientation). Moreover, because $\phi$ is also equal to its inverse, we have:
$$h_2=\phi h_1^{-1} \phi^{-1} = \sigma h_1 \sigma^{-1}.$$
$$ h_2^{-1}= \phi h_1 \phi^{-1}.$$
$$ h_1^{-1}=\sigma \phi h_1 (\sigma \phi)^{-1}.$$

On the other hand, we have $\phi \alpha_{-1}=\sigma \alpha_0$ (see Figure \ref{h et compagnie}). Since $\alpha_n=h_1^n(\alpha_0)$, according to the third previous equality, we have for all $k\in \Z$:
$$\phi \alpha_{-k-1}=\sigma \alpha_k.$$

If we extend the "complete sequences" vocabulary defined in Section \ref{section1} to the rays which start in the northern hemisphere, for example by adding 'north' or 'south' as first term in the sequence of segments of the ray, we can code the $\phi \alpha_k$'s. Using Corollary \ref{dist} and the equality above it, we can deduce that (see Figure \ref{h et compagnie}):
\begin{itemize}
\item For all $n \geq 2$, every ray in the $(n-2)$-neighborhood of $h_2^n(\phi \alpha_0)=h_2^n(\sigma \alpha_{-1})=\sigma h_1^{n-1}(\alpha_0)$ begins like ${\phi \mathring \alpha_{-2}}=\mathring{\sigma \alpha_1}$;

\item For all $n \geq 2$, every ray in the $(n-2)$-neighborhood of $h_2^{-n}(\phi \alpha_0)=\phi(h_1^n\alpha_0)$ begins like $ {\phi \mathring \alpha_{2}}$;

\item For all $n \geq 2$, every ray in the $(n-2)$-neighborhood of $h_1^{-n}(\alpha_0)=\sigma \phi h_1^n \phi \sigma \alpha_0=\sigma \phi h_1^n \alpha_{-1}$ begins like $\mathring \alpha_{-2} = \sigma \phi \mathring \alpha_1$.
\end{itemize}

\begin{figure}[h]
\centering
\def\svgwidth{0.95\textwidth}
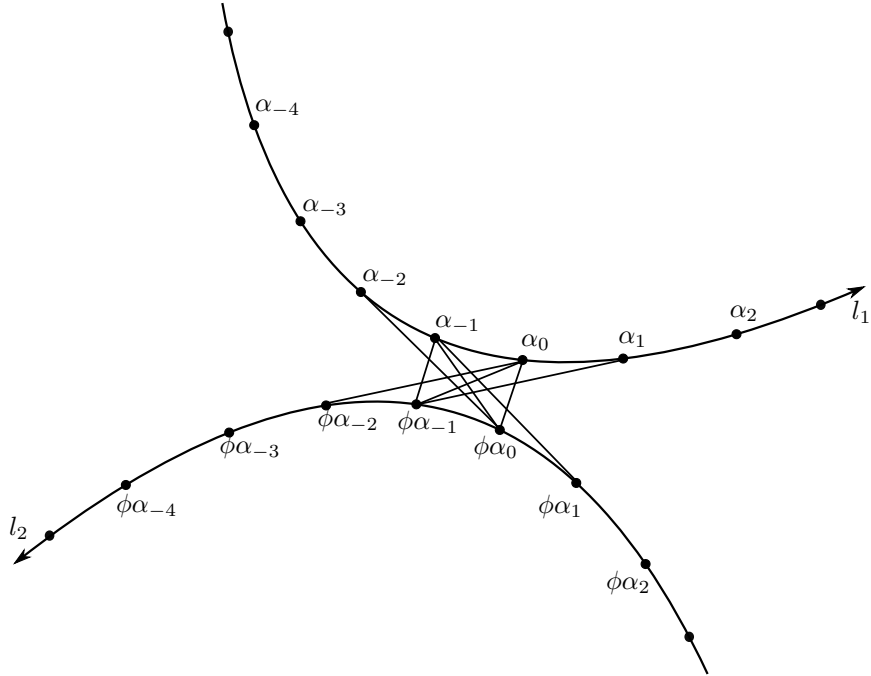
\caption{Axes of $h_1$ and $h_2$: this graph is isometrically embedded in $X_r$.}
\label{figu:axes}
\end{figure}

Finally, for all $n\geq 2$, all the elements in the $(n-2)$-neighborhood of $h_1^n(\alpha_0)$, $h_2^n(\phi \alpha_0)$, $h_2^{-n}(\phi \alpha_0)$ and $h_1^{-n}(\alpha_0)$ respectively, begins like $\mathring \alpha_2$, ${\phi \mathring \alpha_{-2}}$, ${\phi \mathring \alpha_{2}}$ and $\mathring \alpha_{-2}$ respectively. However, $\mathring \alpha_2$, ${\phi \mathring \alpha_{-2}}$, ${\phi \mathring \alpha_{2}}$ and $\mathring \alpha_{-2}$ do not have mutually disjoint representatives: these four neighborhoods are disjoint in $X_r$. Thus the axes $l_1$ and $l_2$ of $h_1$ and $h_2$ are such that the distance between two half-axes is unbounded.
\end{proof}

\begin{remark} More precisely, we have for all $|n|,|m| \geq 2$ (see Figure \ref{figu:axes}): $$d(h_2^n(\phi \alpha_0), h_1^m(\alpha_0))\geq |n|+|m|-1.$$ 
\end{remark}

\section{Example of a loxodromic element with vanishing scl}

Danny Calegari proved that the elements of $\Gamma$ which have a bounded orbit on the ray graph have vanishing $\scl$ (see \cite{blog-Calegari}). Let us show here that the converse is not true.

 \begin{prop}\label{ex-scl_nulle-hyp}
There exists $g\in \Gamma$ which has a loxodromic action on the ray graph and such that $\scl(g)=0$.
 \end{prop}

 \begin{proof} 
We consider again $h_1$ and $h_2$, the two loxodromic elements of $\Gamma$ described in the proof of Theorem \ref{dim infinie}: $h_1$ is the element $h$ previously defined, and $h_2=\phi h_1^{-1} \phi^{-1}$, where $\phi$ is the isotopy class of the rotation of angle $\pi$ around infinity. Let $g:=h_2h_1$ (see Figure \ref{homeo-g}). Then $g$ is conjugated to its inverse (because  $\phi=\phi^{-1}$), and thus $\scl(g)=0$. 

Let us show that $g$ has a loxodromic action on the ray graph. To this end, we construct a geodesic half-axis $(\gamma_k)_k$ in the ray graph, on which $g$ acts by translation (as we did before with $(\alpha_k)_k$ to prove that $h$ is loxodromic).

\begin{figure}[!h]
\labellist
\small\hair 2pt
\pinlabel $h_1$ at 235 130
\pinlabel $h_2$ at 235 40
\endlabellist
\centering
\vspace{0.3cm}
\includegraphics[scale=0.9]{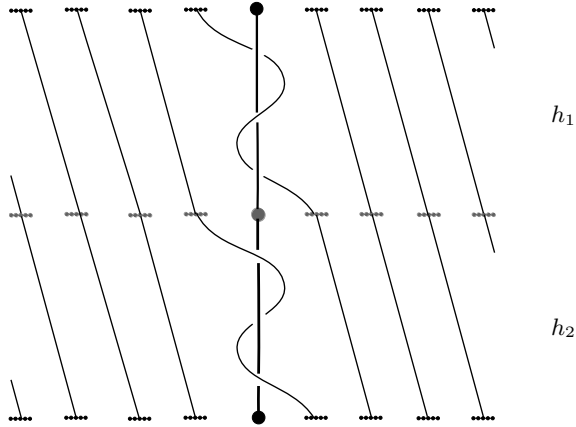}
\caption{The element $g:=h_2h_1$.}
\label{homeo-g}
\end{figure}

\paragraph{Definition of $(\gamma_k)_k$.}
The sequence $(\gamma_k)_k$ is defined similarly to $(\alpha_k)_k$, except that:
\begin{itemize}
\item to define $\alpha_{k+1}$ from $\alpha_k$, we used to follow $\alpha_k$ and turn around its Cantor-endpoint \emph{counterclockwise};
\item to define $\gamma_{k+1}$ from $\gamma_k$, we will follow $\gamma_k$ and turn around its Cantor-endpoint alternately \emph{clockwise} and \emph{counterclockwise}  (see Figures \ref{gamma012} and~\ref{gammak}).
\end{itemize}

\begin{figure}[!h]
\labellist
\small\hair 2pt
\pinlabel $\gamma_{0}$ at 59 70
\pinlabel $\gamma_{1}$ at 303 45
\pinlabel $\gamma_{2}$ at 188 -10
\endlabellist
\centering
\vspace{0.3cm}
\includegraphics[scale=0.75]{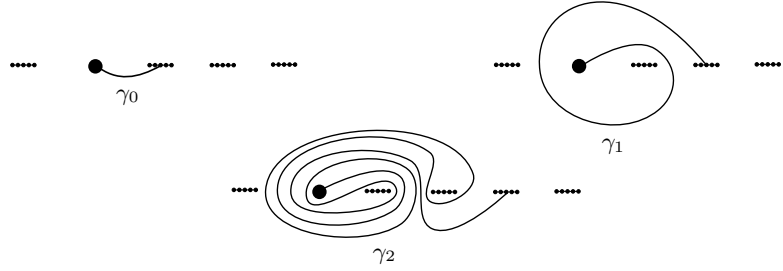}
\vspace{0.5cm}
\caption{Rays $\gamma_0$, $\gamma_1$ and $\gamma_2$.}
\label{gamma012}
\end{figure}

\vspace{0.3cm}
\noindent More precisely, we define the sequence $(\gamma_k)_{k\geq0}$ by induction as follow:
\begin{itemize}
\item  $\gamma_0:=\alpha_0$ is the isotopy class of the segment $s_0$ whose endpoints are $\infty$ and ${p}_0$.

\item For all odd $k\geq 1$ (\emph{clockwise turn}): to draw $\gamma_{k+1}$, we start at $\infty$, we follow $\gamma_k$ until its Cantor-endpoint ${p}_k$ (in a tubular neighborhood of $\gamma_k$), we turn clockwise around this point, crossing the two adjacent segments, \emph{first $s_{k+1}$, and then $s_{k}$}, we follow $\gamma_k$ again in a tubular neighborhood, we turn around infinity by crossing $s_{0}$ and then $s_{-1}$, we follow $\gamma_k$ for the last time in a tubular neighborhood until its Cantor-endpoint, and we go to the point ${p}_{k+1}$ without crossing the equator.

\item For all even $k\geq 1$ (\emph{counterclockwise turn}): to draw $\gamma_{k+1}$, we start at $\infty$, we follow $\gamma_k$ until its Cantor-endpoint ${p}_k$ (in a tubular neighborhood of $\gamma_k$), we turn counterclockwise around this point, crossing the two adjacent segments, \emph{first $s_{k}$, and then $s_{k+1}$}, we follow $\gamma_k$ again in a tubular neighborhood, we turn around infinity by crossing $s_{-1}$ and then $s_{0}$, we follow $\gamma_k$ for the last time in a tubular neighborhood until its Cantor-endpoint, and we go to the point ${p}_{k+1}$ without crossing the equator.

\end{itemize}

\begin{figure}[!h]
\labellist
\small\hair 2pt
\pinlabel $\gamma_{2n-1}$ at -15 100
\pinlabel $\gamma_{2n}$ at -10 20
\pinlabel $K_{2n-1}$ at 336 89
\pinlabel $K_{2n}$ at 371 25
\pinlabel $1$ at 299 83
\pinlabel $3$ at 256 77
\pinlabel $3^{2n-5}$ at 174 77
\pinlabel $3^{2n-4}$ at 132 85
\pinlabel $3^{2n-3}$ at 87 79

\pinlabel $3^{2n-2}$ at 30 75
\pinlabel $1$ at 301 123
\pinlabel $3$ at 253 113
\pinlabel $3^{2n-5}$ at 178 115
\pinlabel $3^{2n-4}$ at 134 124
\pinlabel $3^{2n-3}$ at 91 113
\pinlabel $3^{2n-1}$ at 30 -7
\pinlabel $1$ at 336 -4
\endlabellist
\centering
\vspace{0.4cm}
\includegraphics[scale=0.9]{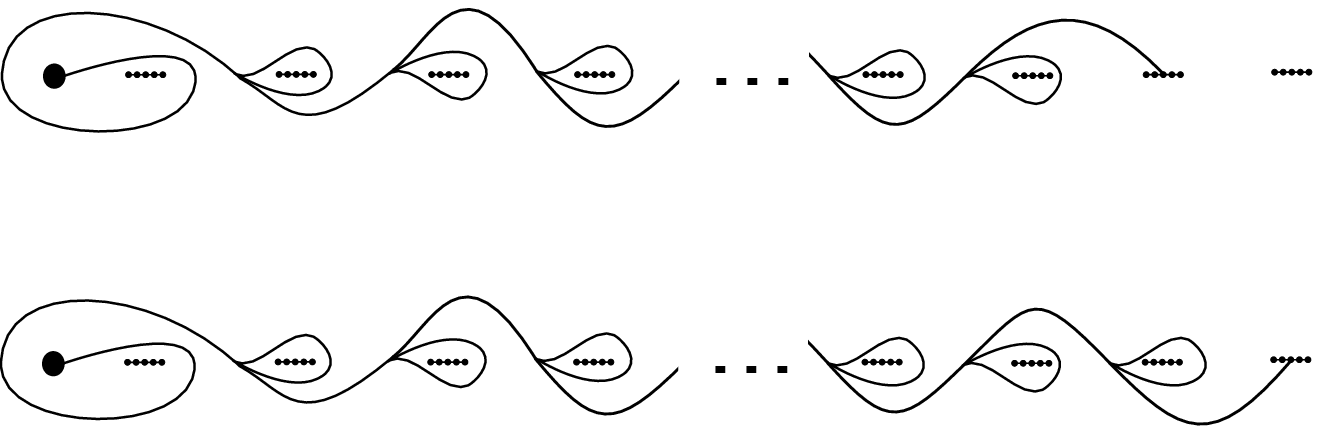}
\vspace{0.4cm}
\caption{Graphs representing $\gamma_k$ for $k$ odd (above) and $k$ even (below).}
\label{gammak}
\end{figure}

Given the similarities between the constructions of $(\alpha_k)_k$ and $(\gamma_k)_k$, we can use the same arguments than the ones used in Section \ref{section1}: in particular, Lemma \ref{lemme distance}, its Corollary \ref{dist} and Proposition \ref{alpha_k geod}. We deduce from these results than $(\gamma_k)_{k\in \N}$ is geodesic.

\paragraph{The element $g$ acts by translation on $(\gamma_k)_k$.} 
By induction, using the graphs representing the rays (as we did before to prove that $h$ is loxodromic), we can see that $g^n(\gamma_0)=\gamma_{2n}$ for all $n\in \N$: see Figure \ref{diag-g-gamma0} for $n=0$, Figure \ref{diag-g-gamma2} for $n=1$, and Figure \ref{diag-g-gammak} for the general case.
\end{proof}

\begin{figure}[!h]
\labellist
\small\hair 2pt
\pinlabel $\gamma_0$ at 101 349
\pinlabel $\gamma_2$ at 97 -9
\pinlabel $h_1$ at 235 328
\pinlabel $h_2$ at 278 154
\pinlabel $g$ at 391 249
\pinlabel $3$ at 397 -8
\pinlabel $1$ at 473 -5
\pinlabel $1$ at 473 39
\endlabellist
\centering
\vspace{0.3cm}
\includegraphics[scale=0.7]{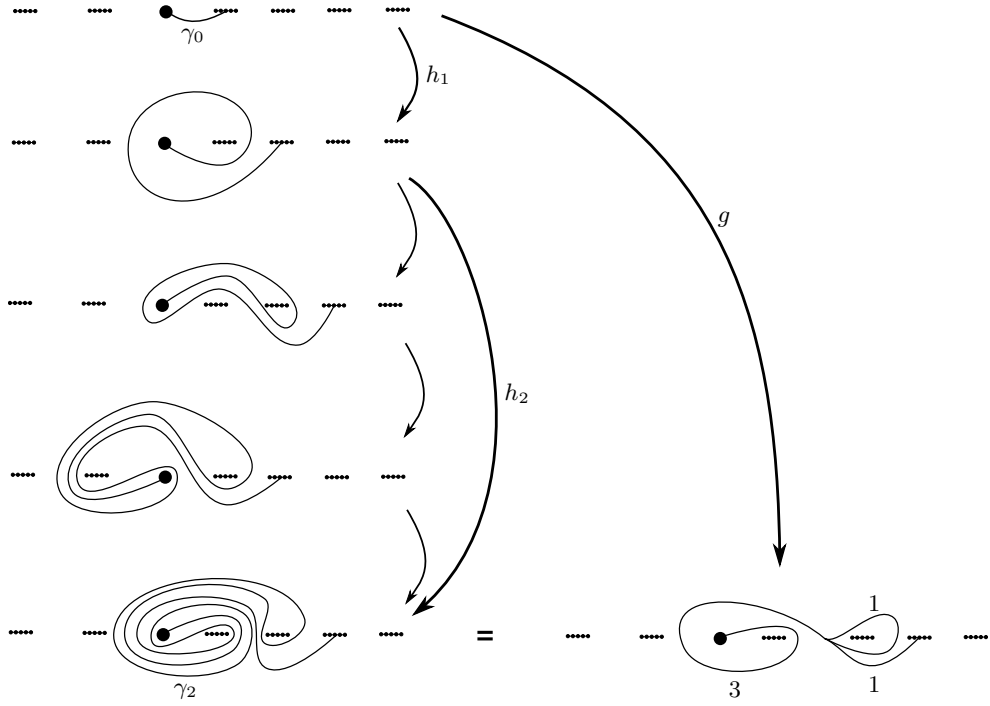}
\vspace{0.1cm}
\caption{Graph representing the image of $\gamma_0$ by $g$, where we decomposed the action of $h_2$ into three smaller parts, as we did before for $h$. The two graphs below represent the same ray: \mbox{$g(\gamma_0)=\gamma_2$.}}
\label{diag-g-gamma0}
\end{figure}

\begin{figure}[!h]
\labellist
\small\hair 2pt
\pinlabel $h_1$ at 385 122
\pinlabel $h_2$ at 385 45
\pinlabel $3$ at 91 128
\pinlabel $1$ at 178 141
\pinlabel $3^2$ at 89 61
\pinlabel $3$ at 180 92
\pinlabel $1$ at 228 78
\pinlabel $3^3$ at 99 -9
\pinlabel $3^2$ at 184 17
\pinlabel $3$ at 230 15
\pinlabel $1$ at 278 9

\pinlabel $\gamma_2$ at -10 152
\pinlabel $\gamma_4$ at -10 16
\endlabellist
\centering
\vspace{0.5cm}
\includegraphics[scale=0.9]{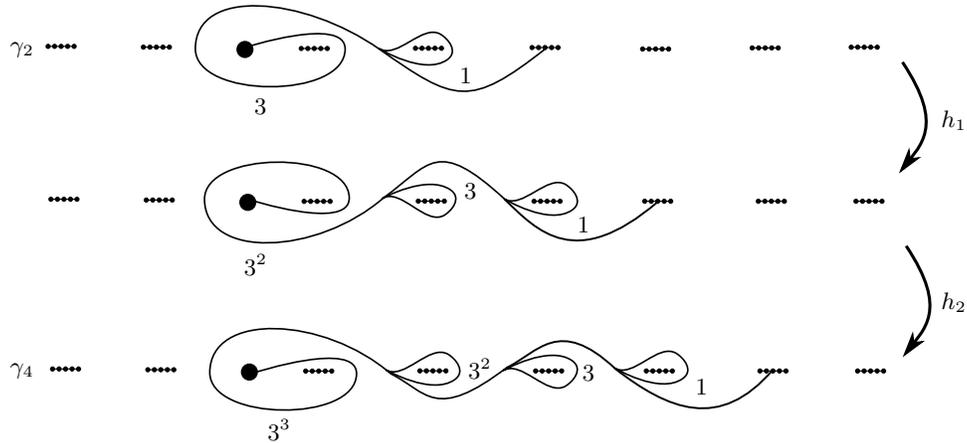}
\vspace{0.5cm}
\caption{Action of $g$ on $\gamma_2$.}
\label{diag-g-gamma2}
\end{figure}

\begin{figure}[!h]
\labellist
\small\hair 2pt
\pinlabel $\gamma_{2n}$ at 6 183
\pinlabel $\gamma_{2n+2}$ at 9 13
\pinlabel $3^{2n-1}$ at 58 169
\pinlabel $3^{2n}$ at 58 84
\pinlabel $3^{2n+1}$ at 58 -7
\pinlabel $1$ at 327 165
\pinlabel $1$ at 367 85
\pinlabel $1$ at 415 -9
\pinlabel $K_{2n}$ at 357 203
\pinlabel $K_{2n+1}$ at 398 118
\pinlabel $K_{2n+2}$ at 442 31
\pinlabel $h_1$ at 518 153
\pinlabel $h_2$ at 520 64
\pinlabel $3^{2n-2}$ at 127 174
\pinlabel $3^{2n-1}$ at 125 93
\pinlabel $3^{2n}$ at 127 -2
\pinlabel $3^{2n-2}$ at 178 85
\pinlabel $3^{2n-2}$ at 225 -5
\pinlabel $3^{2n-1}$ at 175 3

\endlabellist
\centering
\vspace{0.5cm}
\includegraphics[scale=0.7]{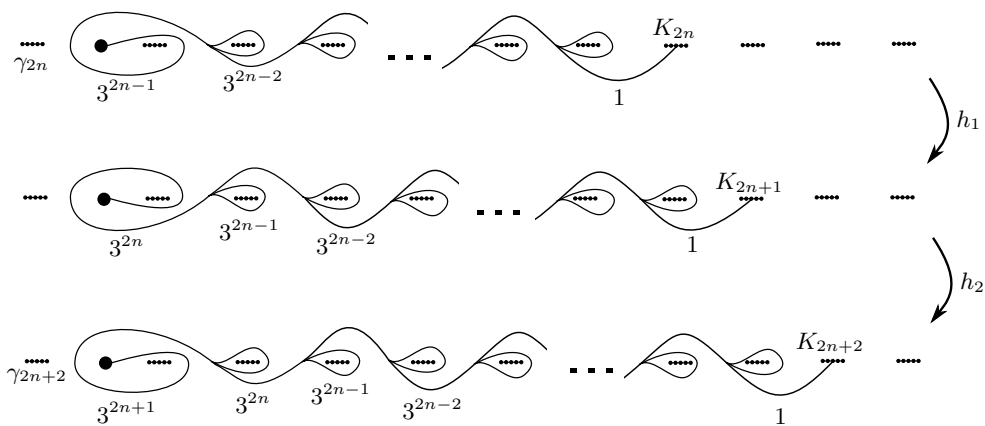}
\vspace{0.1cm}
\caption{Action of $g$ on $\gamma_{2n}$ : $g(\gamma_{2n}) = \gamma_{2n+2}$.}
\label{diag-g-gammak}
\end{figure}

\bibliographystyle{alpha}
\bibliography{ref}

\end{document}